[1994/06/01]

\documentclass[12pt,reqno,intlimits,a4paper,english]{amsart}

\usepackage{pstricks,graphicx,pst-3dplot}
\usepackage{calc}
\usepackage{babel}
\usepackage{amssymb}

\numberwithin{equation}{section}
\newcounter{hours}\newcounter{minutes}
\newcommand{\hourandminute}{\setcounter{hours}{\time/60}%
\setcounter{minutes}{\time-\value{hours}*60}%
\thehours.\theminutes}
\newcommand{\Versione}{\jobname\ \today\ \hourandminute}
\newlength{\Indent}
\newlength{\Parskip}
\theoremstyle{plain}
\newtheorem{thm}{Theorem}[section]     
\newtheorem{lemma}[thm]{Lemma}
\newtheorem{prop}[thm]{Proposition}

\theoremstyle{remark}

\newtheorem{Remark}[thm]{Remark}
\newenvironment{remark}{\begin{Remark}}{\qed\end{Remark}}

\theoremstyle{definition}
\newtheorem{Defin}[thm]{Definition}
\newenvironment{defin}{\begin{Defin}}{\qed\end{Defin}}

\DeclareMathOperator{\Div}{div}

\newcommand{\win}{{\textsc{w}_o}}
\newcommand{\Iion}{I_{\textrm{ion}}}
\newcommand{\aaa}{a}
\newcommand{\bbb}{b}

\newcommand{\RN}{\Bbb{R}^{N}}
\newcommand{\R}{\Bbb{R}}

\newcommand{\di}{\,\text{\rmfamily\upshape d}}

\newcommand{\Om}{\varOmega}
\newcommand{\eps}{\varepsilon}

\newcommand{\vi}{\text{\rmfamily\upshape v}}

\newcommand{\CC}{\mathcal{C}}
\newcommand{\ZZ}{\Bbb{Z}}

\newcommand{\Z}{\ZZ}
\def\X{{\mathcal X}}

\newcommand{\dfint}{\sigma_{1}^B}
\newcommand{\dfinte}{\sigma^{B,\eps}_1}
\newcommand{\dfout}{\sigma_2^B}
\newcommand{\dfoute}{\sigma_{2}^{B,\eps}}
\newcommand{\dfdise}{\sigma^{D,\eps}}
\newcommand{\dfdis}{\sigma^{D}}

\newcommand{\capuno}{\alpha}
\newcommand{\capdue}{\beta}
\newcommand{\ubuno}{u^{B,\eps}_1}
\newcommand{\ubdue}{u^{B,\eps}_2}
\newcommand{\ud}{u^{D,\eps}}

\newcommand{\dfboth}{\sigma}
\newcommand{\dfbothe}{\sigma^\eps}

\newcommand{\Per}{E}
\newcommand{\Omint}{\Om^{D,\eps}}
\newcommand{\Omout}{\Om^{B,\eps}}
\newcommand{\Memb}{\varGamma^{\eps}}
\newcommand{\Perint}{\Per^{\textup{D}}}
\newcommand{\Perout}{\Per^{\textup{B}}}
\newcommand{\Permemb}{\varGamma}

\newcommand{\wto}{\rightharpoonup}
\newcommand{\cell}{Y}
\newcommand{\celloc}[1]{\cell_{\eps}#1}
\newcommand{\genfun}{w}
\newcommand{\hatfun}{\widehat{\genfun}}

\newcommand{\genfune}{\genfun_{\eps}}

\newcommand{\intepart}[1]{\left[#1\right]}
\newcommand{\macropart}[2]{\left[\frac{#1}{#2}\right]_{\cell}}

\newcommand{\micropart}[2]{\left\{\frac{#1}{#2}\right\}_{\cell}}

\newcommand{\intOset}{\widehat{\Om}_{\eps}}
\newcommand{\intspacetime}{\Lambda^{\eps}_T}
\newcommand{\const}{\gamma}
\newcommand{\unfop}{\mathcal{T}_{\eps}}
\newcommand{\btsunfop}{\mathcal{T}^b_{\eps}}
\newcommand{\average}{\mathcal{M}^\eps}
\newcommand{\saverage}{\mathcal{M}_{\cell}}

\newcommand{\testb}{\varphi_B}
\newcommand{\testd}{\varphi_D}
\newcommand{\testduno}{\varphi_D^1}
\newcommand{\testddue}{\varphi_D^2}
\newcommand{\ve}{v^\eps}
\newcommand{\hve}{\hat v^\eps}
\newcommand{\hvi}{\hat {\vi}}
\newcommand{\ue}{u^\eps}
\newcommand{\hue}{\hat u^\eps}
\newcommand{\espo}{\ell}
\newcommand{\nue}{\nu_\eps}
\def\we{{\widetilde w^\eps}}
\newcommand{\Ahomuno}{A^*}
\newcommand{\Ahomdue}{A^{hom}}
\newcommand{\wAhomdue}{\widehat A^{hom}}
\newcommand{\Ahomtre}{\widetilde A}

\makeindex
\begin{document}

\title
{Homogenization of a modified bidomain model involving imperfect transmission}
\author{M. Amar$^\dag$ -- D. Andreucci$^\dag$ -- C. Timofte$^\S$\\
\hfill \\
$^\dag$Dipartimento di Scienze di Base e Applicate per l'Ingegneria\\
Sapienza - Universit\`a di Roma\\
Via A. Scarpa 16, 00161 Roma, Italy
\\ \\
$^\S$University of Bucharest\\
Faculty of Physics\\
P.O. Box MG-11, Bucharest, Romania
}

\begin{abstract}
We study, by means of the periodic unfolding technique, the homogenization of a modified bidomain model, which describes the propagation of the action potential in
the cardiac electrophysiology. Such a model, allowing the presence of pathological zones in the heart, involves various geometries
and non-standard transmission conditions on the interface between the healthy and the damaged part of the cardiac muscle.
\medskip


  \textsc{AMS-MSC:} 35B27, 35Q92, 35K20
  \bigskip

  \textbf{Acknowledgments}: The first author is member of the \emph{Gruppo Nazionale per l'Analisi Matematica, la Probabilit\`{a} e le loro Applicazioni} (GNAMPA) of the \emph{Istituto Nazionale di Alta Matematica} (INdAM).
The second author is member of the \emph{Gruppo Nazionale per la Fisica Matematica} (GNFM) of the \emph{Istituto Nazionale di Alta Matematica} (INdAM). The last author wishes to thank \emph{Dipartimento di Scienze di Base e Applicate per l'Ingegneria} for the warm hospitality and \emph{Universit\`{a} ``La Sapienza" of Rome} for the financial support.
\end{abstract}

\maketitle



\section{Introduction}\label{s:introduction}
In the last years, the mathematical modeling of the electrical activity of the heart was a topic of major interest in biomedical research. A better understanding  of the complex bioelectrical processes involved in the activity of the heart is a key issue in order to find new drugs and diagnostic techniques, being well-known that a huge part of the heart diseases is produced by some disorders of its electrical activity.

One of the most well-known mathematical models in cardiac electrophysiology is the so-called {\it bidomain model} (see, e.g.,
\cite{KS,NK} and the references therein; see, also, the references quoted in \cite[Introduction]{Collin:Imperiale:2018}). In this model, at a macroscopic scale, the electric activity of the heart is governed by a system of two degenerate reaction-diffusion partial differential equations for the averaged intra-cellular and, respectively, extra-cellular electric potentials, along with the transmembrane potential, coupled in a nonlinear manner to ordinary differential equations describing the dynamics of the ion channels. In such a model, the cardiac tissue is represented, at a macroscopic scale, despite its discrete cellular structure, as the superposition of two continuous media, called the {\it intra-cellular} and, respectively, the {\it extra-cellular domain}, coexisting at each point of the heart tissue and connected through a distributed continuous cellular membrane.

 Several ionic models are considered in the literature for describing the cellular membrane dynamics, starting with the famous Hodgkin-Huxley formalism and continuing with more and more complex models
 (see, for instance, \cite{Collin:Imperiale:2018,Davidovic:2016,KS}).
 The well-posedness of the bidomain model has been studied, for different nonlinear ionic models and by using different techniques, by several authors (see, for instance, \cite{BK,Bourgault:Coudiere:Pierre:2009,Nagumo:Arimoto:Yoshizawa:1962,Pennacchio:Savare:Franzone:2005,Veneroni:2009}).

The bidomain model can be obtained as the homogenized version of an electrical conduction problem posed at the scale of single cells
(see, among others, \cite{Amar:Andreucci:Bisegna:Gianni:2006a,Amar:Andreucci:Bisegna:Gianni:2013,Amar:Debonis:Riey:2019,
Bourgault:Coudiere:Pierre:2009,Collin:Imperiale:2018,Grandelius:2019,NK,Nagumo:Arimoto:Yoshizawa:1962,V1}).
However, the problem for the single cell can be, in turn, obtained as the upscaling of the ionic model
(see, in the context of calcium dynamics,
\cite{Goel:Sneyd:Friedman:2006,Graf:Peter:Sneyd:2014,Higgins:Goel:Puglisi:Bers:Cannell:Sneyd:2007,Timofte:2017}).

The bidomain model is widely recognized as being the standard model used in cardiac electrophysiology
for describing the propagation of the action potential in a perfectly healthy cardiac tissue, but it is no longer valid in pathological situations, in which the heart contains electrically passive zones of fibrotic tissue, collagen or fat, as observed for instance in scars, inflammations, ischemic or rheumatic heart diseases, etc. Thus, it is important to find a suitable mathematical model that accounts for the presence of pathological zones in the heart.
Such a model was proposed in \cite{CDP2,Davidovic:2016};
it takes into account the presence in the cardiac tissue of damaged zones, called {\it diffusive inclusions} and assumed to be passive electrical conductors.

In the above mentioned papers, at the mesoscopic scale,
the heart tissue is considered to be a periodic structure obtained by inserting in a healthy tissue a set of periodically distributed diffusive inclusions. From the mathematical point of view, we have a bidomain system coupled with a diffusion equation. More precisely,
the model consists of a degenerate reaction-diffusion
system of partial differential equations modeling the intra-cellular and, respectively, the extra-cellular electric potentials of the healthy cardiac tissue, coupled with an elliptic equation for the passive regions and with an ordinary differential equation describing the cellular membrane dynamics.
A similar model arises also in coupling the torso to the heart (see, e.g., \cite{Boulakia:2015,Boulakia-2007-1,Veneroni:2009}).
The above model is, indeed, a {\it mesoscopic one}, the diffusive inclusions being considered at an {\it intermediate scale} in between the cardiac cell scale and the heart tissue scale.

The mentioned modifications assume a perfect electrical coupling between the healthy part of the heart and the damaged tissue.
More general conditions for the heart-torso coupling were proposed in \cite{Boulakia-2007-1}
and investigated through numerical simulations in \cite{Boulakia:2015,Boulakia-2010}, in order to take into account the possible capacitive and resistive effects of the pericardium. We investigate these more general conditions in the context of the bidomain model with diffusive inclusions, where
the appropriate interface behaviour, up to our knowledge, is still not well understood.
The well-posedness of such a problem was addressed for the first time, as far as we know,
in \cite{Amar:Andreucci:Timofte:2019A}.
Here, we rigorously investigate, in more general geometries than the ones considered in \cite{CDP2}, the homogenization of the problem with
non-standard interface conditions, this being, in fact, the main novelty of our paper.

To achieve our goal, we use the periodic homogenization unfolding technique. The limit problems highly depend on the scaling of the
imperfect transmission across the membrane and on the geometry of the domain.
The influence of the diffusive zone is captured in the limit in several different ways.
In particular, for some special scalings and geometries,
we obtain a bidomain system with memory effects (see Theorem \ref{t:t1} and Remark \ref{r:r9})
or a kind of tridomain model (see Theorem \ref{t:t1a1} and Remark \ref{r:r11}).

We point out again that
our model generalizes the modified bidomain one with diffusive inclusions and perfect transmission conditions considered in \cite{CDP2,Davidovic:2016},
the original model being recovered by suitably rearranging the
parameters appearing in equation \eqref{eq:Circuit}.
We, finally, remark that homogenization techniques can be applied to improve the design of biomedical devices used in 
heart problems (\cite{Gaudiello:Lenczner:2020} and the references therein).

\medskip
The paper is organized as follows: Section \ref{s:threeD_problem} is devoted to the geometrical and functional setting and to the
introduction of the microscopic problem. In Section \ref{s:time_unfolding}, we introduce the time-depending unfolding operator
and some of its properties.
In Section \ref{s:homog}, we state and prove our main homogenization results.

\section{The microscopic problem}\label{s:threeD_problem}

\subsection{Geometrical setting}\label{ss:geometric}
The typical periodic geometrical setting is displayed in Figure~\ref{fig:omega} and Figure~\ref{fig:3d}.
Here we give, for the sake of clarity, its detailed formal definition.
Let $N\geq 3$. Let $\Om$ be an open connected bounded subset of $\RN$
and introduce a periodic open subset $\Per$
of $\RN$, such that $\Per+z=\Per$ for all $z\in\ZZ^{N}$.
We assume that $\Om$ and $E$ are of class $\CC^\infty$,
though this assumption can be weakened.

We employ the notation $Y=(0,1)^{N}$ and
$\Perint=\Per\cap Y$, $\Perout=Y\setminus\overline{\Per}$,
$\Permemb=\partial\Per\cap \overline Y$.
We assume that $\Perout$ is connected,
while $\Perint$ may be connected or not. Moreover,
we stipulate that $|\Permemb\cap\partial Y|_{N-1}=0$.

Let $\eps \in (0,1)$ be a small positive parameter, related to the characteristic dimension of the
microstructure and which takes values in a sequence of strictly positive numbers tending to
zero.
We define $\Omint=\Om\cap\eps \Per$,
$\Omout=\Om\setminus\overline{\eps \Per}$, so that
$\Om=\Omint\cup\Omout\cup\Memb$, where $\Omint$ and $\Omout$ are
two disjoint open subsets of $\Om$ and
$\Memb=\partial\Omint\cap\Om=\partial\Omout\cap\Om$.
From the biological point of view, $\Om$ represents the region occupied by the cardiac tissue,
$\Omout$ [respectively, $\Omint$] corresponds to the bidomain
phase [respectively, the damaged part], while
$\Memb$ is the interface between these two regions;
in fact, these definitions are slightly modified below.
We assume also that
$\Omout$ is connected at each step $\eps>0$, while $\Omint$ will be connected or disconnected.
Indeed, we will consider two different cases: in the first one (to which we will
refer as the {\it connected/disconnected case}, see Figure \ref{fig:omega}),
we will assume that $\Permemb\cap \partial Y=\emptyset$. We also
stipulate that all the cells which intersect $\partial\Om$ do not contain any inclusion,
so that, for all $\eps>0$, $\partial\Om\cap\partial\Omint\not=\emptyset$ and, moreover, ${\rm dist}(\Memb,\partial\Om)\geq c\eps$,
where $c$ is a suitable strictly positive constant.

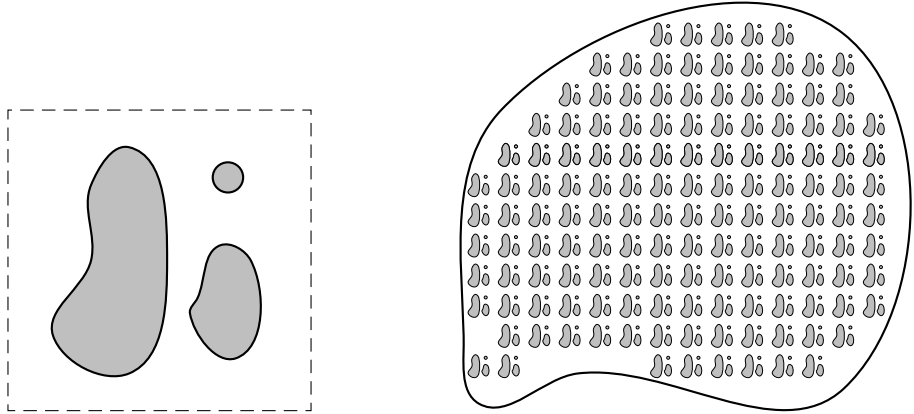
\begin{figure}[htbp]%
\begin{center}%
\begin{pspicture}(12,6)
\rput(0,1){
\psframe[linewidth=0.4pt,linestyle=dashed](0,0)(4,4)
\rput[l](-0.4,0){\psccurve[fillstyle=solid,fillcolor=lightgray](2,.5)(2.5,2)(2,3.5)(1.5,3)(1.5,2)(1,1)}
\psccurve[fillstyle=solid,fillcolor=lightgray](3,.7)(3.2,2)(2.8,2.2)(2.5,1.5)(2.4,1.3)
\pscircle[fillstyle=solid,fillcolor=lightgray](2.9,3.1){.2}
}
\newcommand{\minicell}{\scalebox{0.1}{
\rput[l](-0.4,0){\psccurve[fillstyle=solid,fillcolor=lightgray](2,.5)(2.5,2)(2,3.5)(1.5,3)(1.5,2)(1,1)}
\psccurve[fillstyle=solid,fillcolor=lightgray](3,.7)(3.2,2)(2.8,2.2)(2.5,1.5)(2.4,1.3)
\pscircle[fillstyle=solid,fillcolor=lightgray](2.9,3.1){.2}
}}
\rput(6,1){
\multirput(0,0.4)(0.4,0){2}{\minicell}%
\multirput(2.4,0.4)(0.4,0){6}{\minicell}%
\multirput(0.4,0.8)(0.4,0){12}{\minicell}%
\multirput(0,1.2)(0,0.4){5}{\multirput(0,0)(0.4,0){14}{\minicell}}%
\multirput(0.4,3.2)(0.4,0){13}{\minicell}%
\multirput(0.4,3.2)(0.4,0){13}{\minicell}%
\multirput(0.8,3.6)(0.4,0){12}{\minicell}%
\multirput(1.2,4)(0.4,0){10}{\minicell}%
\multirput(1.6,4.4)(0.4,0){9}{\minicell}%
\multirput(2.4,4.8)(0.4,0){5}{\minicell}%
\psccurve(0.2,.1)(1.5,.5)(5,.3)(5,5)(.5,4)(0,1)
}

\end{pspicture}%
    \caption{\textsl{Left}: the periodic cell $Y$. $\Perint$ is the shaded
    region and $\Perout$ is the white region.
    \textsl{Right}: the region $\Om$.}
    \label{fig:omega}
  \end{center}
\end{figure}

In the second case (to which we will refer as the {\it connected/connected case},
see Figure \ref{fig:3d}) we will
assume that $\Perint$, $\Perout$, $\Omint$ and $\Omout$ are connected and, without loss of generality,
that they have Lipschitz continuous boundary.
In this last case, we have that, for all $\eps>0$, both $\partial\Om\cap\partial\Omout\not=\emptyset$ and
$\partial\Om\cap\partial\Omint\not=\emptyset$. Moreover, we suppose that our geometry satisfies all the assumptions
stated in \cite[Section 3.2.1]{Hopker:2016}. More precisely, in order to avoid technicalities, we choose, in this case,
$\eps$ in such a way that the set $\eps^{-1}\Om$ can be represented as a finite union of axis-parallel cuboids with corner coordinates in $\Z^N$.

\begin{figure}[htbp]%
\begin{center}%
\begin{pspicture*}(-6,0)(6,6)
\psset{Alpha=30,Beta=20,unit=0.8cm}
\rput(0,2.5){
\pstThreeDCoor[xMin=0,xMax=5,yMin=0,yMax=5,zMin=0,zMax=5,,drawing=false,linecolor=black]
\pstThreeDLine(0,0,0)(0,0,5)
\pstThreeDLine(0,0,0)(0,5,0)
\pstThreeDLine(0,0,0)(5,0,0)
\pstIIIDCylinder[
linecolor=gray,
increment=10%
](3,3,0){0.5}{5}
\pstThreeDPut(0,3,3.5){\pstIIIDCylinder[RotY=90,
linecolor=gray,
increment=10%
](0,3,3){0.5}{5}}
\pstThreeDPut(2,0,5){\pstIIIDCylinder[RotX=-90,
linecolor=gray,
,increment=10%
](3,0,3){0.5}{5}
}
\pstThreeDLine(5,5,5)(0,5,5)
\pstThreeDLine(5,5,5)(5,5,0)
\pstThreeDLine(5,5,5)(5,0,5)
\pstThreeDLine(0,0,5)(0,5,5)
\pstThreeDLine(0,0,5)(5,0,5)
\pstThreeDLine(0,5,0)(0,5,5)
\pstThreeDLine(5,0,0)(5,0,5)
\pstThreeDLine(5,5,0)(5,0,0)
\pstThreeDLine(5,5,0)(0,5,0)
}
\end{pspicture*}%
    \caption{The periodic cell $Y$. $\Perint$ is the shaded
    region and $\Perout$ is the white region.}
    \label{fig:3d}
  \end{center}
\end{figure}
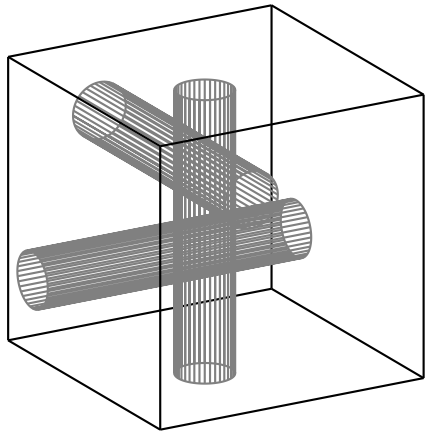

Finally, let $\nu$ denote the normal unit vector to $\Permemb$ pointing
into $\Perout$, extended by periodicity to the whole of $\R^N$, so that $\nue(x)=\nu(x/\eps)$
denotes the normal unit vector to $\Memb$ pointing into $\Omout$.

In the following, by $\const$ we shall denote a strictly positive constant, which may depend
on the geometry and on the other parameters of the problem; $\const$ may vary from line
to line. Moreover, if ${\mathcal G}\subset\R^N$ is
an open set and $T>0$, we set ${\mathcal G}_T={\mathcal G}\times(0,T)$.

\medskip

\subsection{Functional spaces}\label{ss:spaces}

Following \cite{Amar:Andreucci:Timofte:2019A}, we consider the functional spaces
\begin{equation}
\label{eq:space1}
\begin{aligned}
& H^1_{null}(\Omout):= \{w\in H^1(\Omout): \hbox{$w=0$ on $\partial\Omout\cap\partial\Om$, in the sense of traces}\};
\\
& H^1_{null}(\Omint):= \{w\in H^1(\Omint): \hbox{$w=0$ on $\partial\Omint\cap\partial\Om$, in the sense of traces}\}.
\end{aligned}
\end{equation}
Notice that
in the connected/disconnected case
$H^1_{null}(\Omint)=H^1(\Omint)$.

We also define the space
\begin{equation}\label{eq:a1}
\X^1_{0\eps}(\Om):=\{w:\Om\to \R\ :\ w_{\mid_{\Omout}}\in H^1_{null}(\Omout),\
w_{\mid_{\Omint}}\in H^1_{null}(\Omint)\}\,,
\end{equation}
endowed with the norm
\begin{equation}\label{eq:a3}
\Vert w\Vert^2_{\X^1_{0\eps}(\Om)}:= \Vert \nabla w\Vert^2_{L^2(\Omout)}+\Vert w\Vert^2_{H^1(\Omint)}.
\end{equation}
By our assumptions, we have that $\partial\Omout\cap\partial\Om$ is always non-empty, while $\partial\Omint$
can intersect or not the boundary of $\Om$, depending on the geometry.
We recall that, for $w\in \X^1_{0\eps}(\Om)$, the following Poincar\'{e} inequality holds
(see \cite[Proposition 7.1 and Remark 7.1]{Amar:Andreucci:Bisegna:Gianni:2004a}):
\begin{equation}\label{eq:poincare1}
\Vert w\Vert^2_{L^2(\Om)}\leq \const \left(\Vert \nabla w\Vert^2_{L^2(\Omout)}+\Vert \nabla w\Vert^2_{L^2(\Omint)}
+\eps\Vert[w]\Vert^2_{L^2(\Memb)}\right)\,,
\end{equation}
where $[w]=w_{\mid_{\Omout}}-w_{\mid_{\Omint}}$ and the constant $\const$ is independent  of $\eps$.
We point out that the last term is not necessary
in the connected/connected case.
Therefore, an equivalent norm on $\X^1_{0\eps}(\Om)$ is given by
\begin{equation}\label{eq:a7}
\Vert w\Vert^2_{\X^1_{0\eps}(\Om)}\sim \Vert \nabla w\Vert^2_{L^2(\Om)}+\Vert[w]\Vert^2_{L^2(\Memb)};
\end{equation}
again, the last term can be dropped in the connected/connected case.

\subsection{Position of the problem}\label{ss:position}

Let $\capuno,\capdue$ be strictly positive constants and
$\dfint,\dfout,\dfdis$ be $Y$-periodic bounded and symmetric matrices such that
there exist $\const_0,\widetilde\const_0>0$ with
\begin{equation}\label{eq:matrix}
\begin{aligned}
&\const_0|\zeta|^2\leq \dfint(y) \zeta\cdot\zeta\leq\widetilde\const_0|\zeta|^2,\qquad\text{for every
$\zeta\in\RN$ and a.e. $y\in Y$;}
\\
&\const_0|\zeta|^2\leq \dfout(y) \zeta\cdot\zeta\leq\widetilde\const_0|\zeta|^2,\qquad\text{for every
$\zeta\in\RN$ and a.e. $y\in Y$;}
\\
&\const_0|\zeta|^2\leq \dfdis(y) \zeta\cdot\zeta\leq\widetilde\const_0|\zeta|^2,\qquad\text{for every
$\zeta\in\RN$ and a.e. $y\in Y$.}
\end{aligned}
\end{equation}
Moreover, set $\dfinte(x)=\dfint(\eps^{-1}x)$, $\dfoute(x)=\dfout(\eps^{-1}x)$, $\dfdise(x)=\dfdis(\eps^{-1}x)$
for a.e. $x\in\Om$.

As in \cite{Amar:Andreucci:Timofte:2019A}, let us consider a locally Lipschitz continuous function $g:\R^2\to\R$, such that $g(p,1)\geq 0$ and $g(p,0)\leq 0$. The example
we have in mind here is a function of the form
\begin{equation}\label{eq:gating3bis}
g(p,q)= \aaa(p)( q-1) +\bbb(p)q,
\end{equation}
where $\aaa,\bbb:\R\to\R$ are positive, bounded and Lipschitz functions. Notice that the form of $g$ in \eqref{eq:gating3bis} is classical
in this framework (see, for instance, \cite{Veneroni:2009}) and that $g$ is Lipschitz continuous with
respect to $p$ and affine with respect to $q$.
Let $\Iion:\R^2\to\R$ be given by
\begin{equation}\label{eq:ion3}
I_{\textrm{ion}}(p,q)= h_1(p) +h_2(p)q,
\end{equation}
where $h_1,h_2$ are Lipschitz continuous functions and $h_2$ is bounded.
Let $\win\in L^\infty(\Om)$, with $0\leq\win(x)\leq 1$ a.e. in $\Om$ and
$p^\eps\in L^2(\Omout_T)$.
Consider
the gating equation
\begin{alignat}2
\label{eq:gating1}
&\partial_t \widetilde w^\eps_{p^\eps}+g(p^\eps,\widetilde w^\eps_{p^\eps})=0,\qquad &\text{in $\Omout_T$;}
\\
  \label{eq:gating2}
& \widetilde w^\eps_{p^\eps}(x,0)=\win(x),\qquad &\text{in $\Omout$.}
\end{alignat}
Notice that,  by classical results, the previous problem admits a unique solution $\widetilde w_p^\eps\in H^1(0,T;L^\infty(\Omout))$
and, from our assumptions, $0\leq\widetilde w^\eps_{p^\eps}(x,t)\leq 1$ a.e. in $\Omout_T$, since
$0\leq\win(x)\leq 1$ a.e. in $\Omout$.
This is a standard result for ODEs, taking into account that the spatial variable plays here only the role of a parameter (for similar results,
see, for instance, \cite{Collin:Imperiale:2018,CDP2,JPr}).

Moreover, we can write
\begin{equation}\label{eq:gating5}
\widetilde w^\eps_{p^\eps}(x,t)=\win(x)+\int_0^tg(p^\eps,\widetilde w^\eps_{p^\eps})\di \tau\,,\qquad \hbox{for a.e. $x\in \Om$.}
\end{equation}
From the previous assumptions, we can prove that $\Iion$ is a uniformly globally Lipschitz continuous function, i.e.
there exists a strictly positive constant $\const_I$, independent of $\eps$, such that
\begin{equation}
\label{eq:ion1}
\big\Vert \Iion(p^\eps_{1},\widetilde w^\eps_{p^\eps_{1}})-\Iion(p^\eps_{2},\widetilde w^\eps_{p^\eps_{2}})\big\Vert_{L^2(\Omout_T)}
\leq \const_I\Vert p^\eps_{1}-p^\eps_{2}\Vert_{L^2(\Omout_T)}\,,
\end{equation}
due to the uniform Lipschitz dependence of $\widetilde w^\eps_{p^\eps}$ on $p^\eps$
and to the bound $0\leq\widetilde w^\eps_{p^\eps}(x,t)\leq 1$ a.e. in $\Omout_T$.

\begin{remark}\label{r:r4}
Different examples of functions $\Iion$ and $g$ are considered in the literature. We consider here
a Hodgkin-Huxley type model (see \eqref{eq:gating3bis}--\eqref{eq:ion3}), as in \cite{Collin:Imperiale:2018,Veneroni:2009}.
However, we point out that the results obtained in this paper are also valid for a regularized version of the Mitchell-Schaeffer model proposed in
\cite{CDP2} (see, also, \cite{Davidovic:2016}).
For this last model, the ionic current $I_{\textrm{ion}}$ satisfies \eqref{eq:ion1} and
\begin{equation}
\label{eq:ion2}
\Iion(0, \widetilde w_0(x,t))=0\,,\qquad\hbox{a.e. in $\Omout_T$}\,,
\end{equation}
the function $g$ is supposed to be an affine function with respect to $q$ and smooth with respect to $p$.
More precisely, assuming $p$ given, the ionic current $I_{\textrm{ion}}(p,q)$ is defined as being
\[
I_{\textrm{ion}}(p,q)=\frac{1}{\tau_{\textrm{in}}} q p^2(p-1)e^{-(p/p_{th})^2}- \frac{1}{\tau_{\textrm{out}}} p (1+r_{max}e^{-(p_{th}/p)^2})
\]
and the function $g$ is given by
\[
g(p,q)= \left (\frac{1}{\tau_{\textrm{cl}}}+\frac{\tau_{\textrm{cl}}-\tau_{\textrm{op}}}{\tau_{\textrm{cl}}\tau_{\textrm{op}}} q_\infty(p) \right ) (q-q_\infty(p)),
\]
with
\[
q_\infty(p)=1- e^{-(p_{gate}/p)^2}.
\]
Here, all the model parameters are supposed to be positive constants (see, for the interpretation of these constants, \cite{CDP2,
Davidovic:2016}). Further, it is assumed that
$$
0 < \tau_{\textrm{op}} < \tau_{\textrm{cl}}\,\quad
p_{th}\gg p_{gate}\,
\quad\hbox{and}\quad
r_{max} \gg 1\,.
$$
Under the same assumptions we made for the initial data, one can prove that, also for the Mitchell-Schaeffer model,
the ionic function verifies condition \eqref{eq:ion1}; moreover,
$\widetilde w(t, \cdot),\, \partial_t \widetilde w(t, \cdot)\in L^\infty(\Omega^B)$ and $0 < \widetilde w(t, x)\leq  1$
(see \cite[Lemma 6 and Proposition 17]{Davidovic:2016} or
\cite[Proposition 1]{CDP2}).
\end{remark}

We give here a complete formulation of the problem we shall address in this paper
(the operators $\Div$ and $\nabla$ act
only with respect to the space variable $x$).
\medskip

Let $\espo\geq -1$. Assume that $f_1,f_2\in L^2(0,T;H^1_0(\Om))$, $\overline v_{0}\in L^2(\Om)$ and,
for every $\eps>0$, let  $s_{0\eps}\in L^2(\Memb)$ be such that
\begin{equation}\label{eq:a20}
\frac{1}{\eps^\espo}\int_{\Memb} s_{0\eps}^2\di\sigma\leq \const\,,
\end{equation}
for a suitable $\const>0$, independent of $\eps$.

Let us consider the problem for $\ubuno,\ubdue\in L^2(0,T;H^1_{null}(\Omout)),\ud\in L^2(0,T;H^1_{null}(\Omint))$ and
$\we\in H^1(0,T;L^\infty(\Omout))$ given by
\begin{alignat}2
  \label{eq:PDEin}
&\frac{\partial}{\partial t}(\ubuno\!-\ubdue)-\Div(\dfinte \nabla \ubuno)+I_{ion}(\ubuno\!-\ubdue,\we)\!=\!f_1,\ &\text{in $\Omout_T$;}
  \\
  \label{eq:PDEout}
&\frac{\partial}{\partial t}(\ubuno\!-\ubdue)+\Div(\dfoute \nabla \ubdue)+I_{ion}(\ubuno\!-\ubdue,\we)\!=\!f_2,\ &\text{in $\Omout_T$;}
  \\
  \label{eq:PDEdis}
&-\Div(\dfdise \nabla \ud)=0,&\text{in $\Omint_T$;}
\\
\label{eq:flux1}
&  \dfinte\nabla\ubuno\cdot\nue =0,&\text{on $\Memb_T$;}
\\
\label{eq:flux2}
&  \dfoute\nabla\ubdue\cdot\nue -\dfdise\nabla\ud\cdot\nue =0,&\text{on $\Memb_T$;}
\\
  \label{eq:Circuit}
&  \frac{\capuno}{\eps^\espo}\frac{\partial}{\partial t}(\ubdue-\ud)+
\frac{\capdue}{\eps^\espo}  (\ubdue-\ud)=\dfoute\nabla \ubdue\cdot\nue,
&\text{on $\Memb_T$;}
  \\
  \label{eq:BoundData}
&  \ubuno(x,t),\ubdue(x,t),\ud(x,t)=0,&\text{on $\partial\Om\times(0,T)$;}
  \\
  \label{eq:InitData1}
& \ubuno(x,0)-\ubdue(x,0)=\overline v_{0}(x),&\text{in $\Omout$;}
\\
  \label{eq:InitData3}
& \ubdue(x,0)-\ud(x,0)=s_{0\eps}(x),&\text{on $\Memb$,}
\end{alignat}
where $\widetilde w^\eps$ is the solution of the gating problem \eqref{eq:gating1}-\eqref{eq:gating2},
with $p^\eps=\ubuno-\ubdue$.

\begin{remark}[Biological interpretation]\label{r:r3}
The previous system of equations represents the coupling of a standard bidomain model in $\Omout$, for the intra and the extra-cellular
potentials $\ubuno$ and $\ubdue$ of the healthy zone, with a Poisson equation in the diffusive part $\Omint$, for the electrical potential
$\ud$ of the damaged zone. The damaged zone is modeled as a passive conductor, justifying the quasi-static assumption in \eqref{eq:PDEdis},
as done in \cite{CDP2}.
The function $\ubdue-\ubuno$ is the so-called transmembrane potential.
The sources $f_1$ and $f_2$  are the internal and the external current stimulus, respectively.
The coefficients $\dfint,\dfout$ and $\dfdis$ are the conductivities of the two healthy phases and of the damaged one, respectively;
following \cite{CDP2}, where also numerical simulations are considered, we assume that such conductivities are of the same order.
On the other hand, $\capuno$ and $\capdue$ are given parameters related to the capacitive and the resistive behaviour of the interface $\Memb$.
We point out that for the intra-cellular potential $\ubuno$ we assume no flux condition on $\Memb$ (see \eqref{eq:flux1}), while the extra-cellular
potential $\ubdue$ is coupled with the electrical potential $\ud$ of the damaged zone through non-standard imperfect transmission conditions (see \eqref{eq:flux2} and \eqref{eq:Circuit}). Here, the choice of the scaling parameter $\ell$, leading to different macroscopic behaviours,
is related to the speed of the interfacial exchange.
Our system is completed with suitable initial and boundary conditions.
The variable $\widetilde w^\eps$, called the {\it gating variable}, describes the ionic transport through the cell membrane.
The terms $g$ and $I_{ion}$ are nonlinear functions, modeling the membrane ionic currents.

For simplicity, we consider only one gating variable, but our results hold true also for the case in which the gating variable is vector valued.
\end{remark}

Notice that, by setting $\ve=p^\eps=\ubuno-\ubdue$, $\ue=\ubdue$ a.e. in $\Omout_T$, $\ue=\ud$ a.e. in $\Omint_T$,
and denoting by $[\cdot]$ the jump across $\Memb$ of the quantity in the square brackets, i.e.,
$[\ue]=\ubdue-\ud$ and $[\dfbothe\nabla \ue\cdot\nue]=(\dfoute\nabla \ubdue-\dfdise\nabla \ud)\cdot\nue$,
the previous system can be written in the more convenient form
\begin{alignat}2
  \label{eq:PDEinc}
&\frac{\partial \ve}{\partial t}-\Div(\dfinte \nabla \ve)+I_{ion}(\ve,\we)\!=\!f_1+\Div(\dfinte\nabla \ue),\ &\text{in $\Omout_T$;}
  \\
  \label{eq:PDEoutc}
&-\Div((\dfinte+\dfoute) \nabla \ue)\!=\!f_1-f_2+\Div(\dfinte\nabla \ve),\ &\text{in $\Omout_T$;}
  \\
  \label{eq:PDEdisc}
&-\Div(\dfdise \nabla \ue)=0,&\text{in $\Omint_T$;}
\\
\label{eq:flux1c}
&  \dfinte\nabla ( \ve+\ue)\cdot\nue =0,&\text{on $\Memb_T$;}
\\
\label{eq:flux2c}
&  [\dfbothe\nabla \ue\cdot\nue] =0,&\text{on $\Memb_T$;}
\\
  \label{eq:Circuitc}
&  \frac{\capuno}{\eps^\espo}\frac{\partial}{\partial t}[\ue]+\frac{\capdue}{\eps^\espo}  [\ue]=\dfout\nabla \ue\cdot\nue,
&\text{on $\Memb_T$;}
  \\
  \label{eq:BoundDatac}
&  \ve, \ue=0,&\text{on $\partial\Om\times(0,T)$;}
  \\
  \label{eq:InitData1c}
& \ve(x,0)=\overline v_{0}(x),&\text{on $\Omout$;}
\\
  \label{eq:InitData3c}
& [\ue](x,0)=s_{0\eps}(x),&\text{on $\Memb$,}
\end{alignat}
complemented with the gating problem \eqref{eq:gating1}--\eqref{eq:gating2}, where again
$\ubuno-\ubdue$ is replaced by $\ve$.
Clearly, $\ve\in L^2(0,T;H^1_{null}(\Omout))$ and $\ue\in L^2(0,T;\X^1_{0\eps}(\Om))$.
We stress again that, by \eqref{eq:ion1}, the composed function $\Iion(\ve, \widetilde w^\eps)$
is a Lipschitz function with respect to $\ve$.

The weak formulation of the previous problem is given by
\begin{multline*}
-\int_{\Omout_T} \ve\partial_t\testb\di x\di t+\int_{\Omout_T}\dfinte\nabla \ve\cdot\nabla \testb\di x\di t
+\int_{\Omout_T}\dfinte\nabla \ue\cdot\nabla\testb\di x\di t
\\
+\int_{\Omout_T} I_{ion}(\ve,\we)\testb\di x\di t
+\int_{\Omout_T}(\dfinte+\dfoute)\nabla \ue\cdot\nabla \testduno \di x\di t
\end{multline*}
\begin{multline}\label{eq:weak5}
+\int_{\Omout_T}\dfinte\nabla \ve\cdot\nabla \testduno\di x\di t
+\int_{\Omint_T}\dfdise\nabla  \ue\cdot \nabla\testddue\di x\di t
\\
-\frac{\capuno}{\eps^\espo}\int_{\Memb_T}[\ue]\partial_t[\testd]\di\sigma\di t
+\frac{\capdue}{\eps^\espo}\int_{\Memb_T}[\ue][\testd]\di\sigma\di t
\\
=\int_{\Omout_T}f_1\testb\di x\di t+\int_{\Omout_T}(f_1-f_2)\testduno\di x\di t
\\
+\int_{\Omout} \overline v_0\testb(0)\di x
+\frac{\capuno}{\eps^\espo}\int_{\Memb}s_{0\eps}[\testd](0)\di\sigma
\,,
\end{multline}
for every $\testb\in L^2(0,T;H^1_{null}(\Omout))\cap H^1(0,T;L^2(\Omout))$, $\testduno\in L^2(0,T;H^1_{null}(\Omout))$, $\testddue\in L^2(0,T;H^1_{null}(\Omint))$, and $[\testd]\in H^1(0,T;L^2(\Memb))$, with $\testb(T)=0$ and $[\testd](T)=0$.
Here, $[\testd]=\testduno-\testddue$ on $\Memb$ and \eqref{eq:weak5} shall be complemented with the gating problem.
We remark that the notation $\testduno$ for a test function acting in $\Omout$ is slightly counterintuitive, but
it allows us to write concisely the jump of such a function across the interface.

For any $\eps>0$ fixed, by \cite[Theorem 3.6]{Amar:Andreucci:Timofte:2019A}, it follows that the system \eqref{eq:PDEinc}--\eqref{eq:BoundDatac},
complemented with the gating problem \eqref{eq:gating1}--\eqref{eq:gating2}, admits a unique solution $ \ve\in L^2(0,T;H^1_{null}(\Omout))$,
$ \ue\in L^2(0,T;\X^1_{0\eps}(\Om))$ and $\we\in H^1(0,T;L^\infty(\Omout))$, such that
$ \ve\in \CC^0([0,T];L^2(\Omout))$, $ [\ue]\in \CC^0([0,T];L^2(\Memb))$, at least when $\dfinte,\dfoute,\dfdise$ are scalar coefficients
or special matrices as in \cite[Lemma 1]{Bourgault:Coudiere:Pierre:2009}
and \cite[Formula (1)]{K2020} (see, also, \cite{Boulakia:2015} and \cite{Davidovic:2016}).
Even for more general matrices, our homogenization result holds true, if we assume the existence of solutions satisfying
the energy inequality \eqref{eq:energy3} below.

Moreover, by a standard regularization procedure, multiplying \eqref{eq:PDEinc}
by $\ve$, \eqref{eq:PDEoutc} and \eqref{eq:PDEdisc} by $\ue$,
adding the three equations, integrating by parts, using \eqref{eq:flux1c}--\eqref{eq:InitData3c},
moving the integral containing $\Iion$ to the right-hand side, using \eqref{eq:gating3bis}--\eqref{eq:ion1} and
H\"{o}lder and Gronwall inequalities, as in \cite[inequality (2.38)]{Amar:Andreucci:Timofte:2019A}, we get
the following energy estimate:
\begin{multline}\label{eq:energy3}
\sup_{t\in (0,T)}\int_{\Omout}(\ve)^2(x,t)\di x +\int_{\Omout_T}|\nabla \ve+\nabla \ue|^2\di x\di t
+\int_{\Omout_T}|\nabla \ue|^2\di x\di t
\\
+\int_{\Omint_T}|\nabla \ue|^2\di x\di t+\sup_{t\in (0,T)}\frac{1}{\eps^\espo}\int_{\Memb}[\ue]^2(x,t)\di\sigma
+\frac{1}{\eps^\espo}\int_{\Memb_T}[\ue]^2\di \sigma\di t
\\
\leq\const\left(\Vert f_1\Vert^2_{L^2(\Om_T)}+\Vert f_2\Vert^2_{L^2(\Om_T)}+\Vert \overline v_0\Vert^2_{L^2(\Om)}
+\frac{1}{\eps^\espo}\Vert s_{0\eps}\Vert^2_{L^2(\Memb)}+1\right)\,,
\end{multline}
where $\const$ is independent of $\eps$.
Notice that, by \eqref{eq:energy3}, it follows also that
\begin{equation}\label{eq:energy4}
\int_{\Omout_T}|\nabla \ve\vert^2\di x\di t\leq \const\left(\Vert f_1\Vert^2_{L^2(\Om_T)}+\Vert f_2\Vert^2_{L^2(\Om_T)}+
\Vert \overline v_0\Vert^2_{L^2(\Om)}+\frac{1}{\eps^\espo}\Vert s_{0\eps}\Vert^2_{L^2(\Memb)}+1\right).
\end{equation}
Finally, taking into account condition \eqref{eq:a20}, it follows that the right-hand side in \eqref{eq:energy3} and \eqref{eq:energy4}
is uniformly bounded with respect to $\eps$. Therefore, recalling that, both in the connected/connected case and in
the connected/disconnected one, the trace of $\ve$ and $\ue$ on $(\partial\Om\cap\partial\Omout)\times(0,T)$ is null and
using the Poincar\'{e} inequality \eqref{eq:poincare1}, we get
\begin{equation}\label{eq:energy5}
\begin{aligned}
& \Vert \ve\Vert_{L^2(\Omout_T)}+\Vert \nabla\ve\Vert_{L^2(\Omout_T)}\leq\const;
\\
& \Vert \ue\Vert_{L^2(\Omout_T)}+\Vert \nabla\ue\Vert_{L^2(\Omout_T)}\leq\const;
\\
& \Vert \ue\Vert_{L^2(\Omint_T)}+\Vert \nabla\ue\Vert_{L^2(\Omint_T)}\leq\const;
\\
& \sup_{t\in(0,T)}\frac{1}{\eps^\espo}\int_{\Memb_T} [\ue]^2\di \sigma\leq\const.
\end{aligned}
\end{equation}

\section{Time-depending unfolding operator}
\label{s:time_unfolding}

A space-time version of the unfolding operator in a more general framework,
in which also a time-microscale is actually present,  has been introduced
in \cite{Amar:Andreucci:Bellaveglia:2015B} and \cite{Amar:Andreucci:Bellaveglia:2015A},
to which we also refer for a survey on this topic.
However, in the present case, the
time variable does not play any special role and can be treated essentially as
a parameter; hence, the properties of the unfolding operator can be found in \cite{Cioranescu:Damlamian:Donato:Griso:Zaki:2012,
Cioranescu:2018}.

Here, we recall only the definitions and the main convergence results needed in the following.

For $\xi\in\Xi_\eps$, we define
\begin{equation*}
  \Xi_{\eps}
  =
  \left\{\xi\in \mathbb{Z}^N\,: \quad \eps(\xi+Y)\subset\Om\right\}
\end{equation*}
and set
\begin{equation*}
  \intOset
  =
  \text{interior}\left\{\bigcup_{\xi\in \Xi_{\eps}} \eps(\xi+\overline{\cell})   \right\}
  \,,
\qquad
  \intspacetime
  =
  \intOset\times(0,T)
  \,.
\end{equation*}

Denoting by $[r]$ the integer part and by $\{r\}$ the fractional part of $r\in\R$, we define for $x\in\RN$
\begin{equation*}
  \macropart{x}{\eps}
  =
  \Big(
  \intepart{\frac{x_{1}}{\eps}}
  ,
  \dots
  ,
  \intepart{\frac{x_N}{\eps}}
  \Big)
  \qquad\text{and}\qquad
  \micropart{x}{\eps}
  =
  \Big(
  \left\{\frac{x_{1}}{\eps}\right\}
  ,
  \dots
  ,
  \left\{\frac{x_N}{\eps}\right\}
  \Big)
  \,,
\end{equation*}
so that
\begin{equation*}
  x
  =
  \eps\left(\macropart{x}{\eps}+\micropart{x}{\eps}\right)
  \,.
\end{equation*}
Then, we introduce the space cell containing $x$ as $  \celloc{(x)} =  \eps\Big(\displaystyle\macropart{x}{\eps}
  +  \cell  \Big)$.

\begin{defin}\label{d:oldunfop}
For $\genfun$ Lebesgue-measurable on $\Om_T$, the (time-depending) periodic
unfolding operator $\unfop$ is defined as
\begin{equation*}
  \unfop(\genfun)(x,t,y)
  =
  \left\{
    \begin{alignedat}2
      &\genfun \left(\eps\macropart{x}{\eps}+\eps y, t \right)
      \,,
      &\quad&
      (x,t,y)\in \intspacetime\times\cell
      \,,
      \\
      &0 \,,
      &\quad&
      \text{otherwise.}
    \end{alignedat}
  \right.
\end{equation*}
For $\genfun$ Lebesgue-measurable on $\Memb_T$, the (time-depending) boundary unfolding operator $\btsunfop$
is defined as
\begin{equation*}
  \btsunfop(\genfun)(x,t,y)
  =
  \left\{
    \begin{alignedat}2
      &\genfun \left(\eps\macropart{x}{\eps}+\eps y, t \right)
      \,,
      &\quad&
      (x,t,y)\in \intspacetime\times\Permemb \,,
      \\
      &0 \,,
      &\quad&
      \text{otherwise.}
    \end{alignedat}
  \right.
\end{equation*}
\end{defin}

Clearly, for $\genfun_{1}$, $\genfun_{2}$, as in Definition~\ref{d:oldunfop},
\begin{equation}
  \label{eq:unfop_product}
  \unfop(\genfun_{1}\genfun_{2})
  =
  \unfop(\genfun_{1})
  \unfop(\genfun_{2})
\end{equation}
and the same property holds for the boundary unfolding operator.
Note that $\btsunfop(\genfun)$ is the trace of the unfolding operator
on $\intspacetime\times\Permemb$, when both the operators are defined.
\bigskip

We need also an average operator in space.
\begin{defin}  \label{d:local_averages}
  Let $\genfun$ be integrable in $\Om_T$.
  The local (time-depending) space average operator is defined by
 \begin{equation}
   \label{eq:local_s}
   \average(\genfun)(x,t)
   =
   \left\{
   \begin{alignedat}2
     &\frac{1}{\eps^N}
     \int_{\celloc{(x)}}
     \genfun(\zeta,t)
     \di \zeta
     \,,
     &\quad&
     \text{if}
     \,
     (x,t)\in
     \intspacetime\,,
     \\
     &0
     \,,
     &\quad&
     \text{otherwise.}
   \end{alignedat}
   \right.
 \end{equation}
\end{defin}

\begin{remark}
  \label{r:averages}
  From the above definitions, it follows that
  \begin{equation}
    \label{eq:local_s_ii}
    \average(\genfun)(x,t)
    =
    \int_{\cell}
    \unfop(\genfun)(x,t,y)
    \di y
    =
    \saverage(\unfop(\genfun))(x,t)
    \,,
  \end{equation}
  where $\saverage(\unfop(\genfun))$ denotes the mean average of $\unfop(\genfun)$  over $Y$.

 More in general, given an open set $\mathcal{G}\subset\R^N$, we denote by $\mathcal{M}_{\mathcal{G}}(\genfun)$ the mean average of $\genfun$ over
 $\mathcal{G}$.
\end{remark}
\medskip

\begin{prop}
  \label{t:smalleps_grad_weak_conv}
  Let $\genfun\in L^{2}(\Om_T)$. Then,
\begin{equation}\label{eq:convmedia}
\frac{1}{\eps}
   \left[\unfop(\genfun)-\average(\genfun)\right]
    \to
    y^c
    \cdot
    \nabla \genfun\,,
    \quad
    \text{strongly in}
    \,
    L^2(\Om_T\times\cell)\,,
\end{equation}
  where
  \begin{equation*}
    y^c=\left(
      y_1-\frac{1}{2}
      \,,
      y_2-\frac{1}{2}
      \,,
      \cdots
      ,
      y_N-\frac{1}{2}
    \right)
    \,.
  \end{equation*}
  Let $\{\genfune\}$ be a sequence converging weakly to $\genfun$ in
  $L^2\big(0,T;H^1_0(\Om)\big)$. Then, up to a subsequence, there
  exists $\hatfun=\hatfun(x,t,y)\in
  L^2\big(\Om_T;H^{1}_{\#}(\cell))$, with
  $\saverage(\hatfun)=0$, such that, as $\eps\to 0$,
  \begin{align}
    \label{eq:smalleps_grad_weak_conv_i}
    \unfop(\nabla\genfune)
    &\wto
    \nabla\genfun
    +\nabla_y \hatfun
    \,,
    \quad
    \text{weakly in}
    \,
    L^2(\Om_T\times\cell)
    \,,
    \\
    \label{eq:smalleps_grad_weak_conv_ii}
    \frac{1}{\eps}
    \left[\unfop(\genfune)-\average(\genfune)\right]
    &\wto
    y^c
    \cdot
    \nabla \genfun
    +\hatfun
    \,,
    \quad
    \text{weakly in}
    \,
    L^2(\Om_T; H^1_\#(\cell))
    \,.
  \end{align}
  \end{prop}

\medskip

For later use, we set
\begin{multline}\label{nneq:a20}
\X^1_\#(Y):=\{ \hat w=(\hat w^B,\hat\genfun^D)\ :\  \hat\genfun^B=\hat\genfun_{\mid{\Perout}}\in H^1(\Perout),
\\
\hat\genfun^D=\hat\genfun_{\mid{\Perint}}\in H^1(\Perint),\ \hat\genfun\text{ is $Y$-periodic}\}.
\end{multline}
\medskip

\begin{prop}\label{p:comp AB}
Let $\{\genfune\}\subset L^2(0,T;\X^1_{0\eps}(\Om))$ and assume that we are in the connec\-ted/connected geometry.
Assume that there exists $\const>0$ (independent of $\eps$) such that
\begin{equation}\label{stima comp AB}
\int_{\Om_T}|\genfune|^2\di x\di t+\int_{\Om_T}|\nabla \genfune|^2\di x\di t\leq \const\,,
\quad\quad
\forall\eps>0.
\end{equation}
Then, there exists $\genfun\in L^2(\Om_T;\X^1_\#(Y))$, whose
restrictions to $\Perint$ and $\Perout$ satisfy
\begin{equation*}
\begin{aligned}
& \genfun_{\mid_{\Perout}}(x,t,y)=:\genfun^B(x,t)\in L^2(0,T;H^1(\Om))\,,\quad  & \hbox{for a.e. $y\in\Perout$,}
\\
& \genfun_{\mid_{\Perint}}(x,t,y)=:\genfun^D(x,t)\in L^2(0,T;H^1(\Om))\,, \quad & \hbox{for a.e. $y\in\Perint$,}
\end{aligned}
\end{equation*}
and there exists $\hat\genfun\in L^2(\Om;\X^1_\#(Y)/\R)$ such that, up to subsequence,
as $\eps\to 0$, we have
\begin{alignat}{2}
\label{conv comp AB1}& \unfop(\chi_{\Omint}\genfune)\wto \chi_{\Perint}\genfun^D\!\!\!,\ \
\unfop(\chi_{\Omout}\genfune)\wto \chi_{\Perout}\genfun^B\!\!\!,\quad  && \hbox{weakly in } L^2(\Om_T\times Y)\,;
\\
\label{conv comp AB3}& \unfop(\chi_{\Omint}\nabla\genfune)\wto \chi_{\Perint}\left(\nabla\genfun^D+\nabla_y\hat\genfun^D\right)\,,
  && \hbox{weakly in } L^2(\Om_T\times Y)\,;
  \\
\label{conv comp AB4}& \unfop(\chi_{\Omout}\nabla\genfune)\wto \chi_{\Perout}\left(\nabla\genfun^B+\nabla_y\hat\genfun^B\right)\,,
  && \hbox{weakly in } L^2(\Om_T\times Y)\,,
  \end{alignat}
where, for ${\mathcal O}\subseteq \R^N$, $\chi_{\mathcal O}$ denotes the characteristic function of $\mathcal O$.
Moreover, we have also
\begin{equation}\label{eq:a31}
\eps\int_{\Memb_T} [\genfune]^2\di\sigma\di t\leq 2\eps\int_{\Memb_T} \left(|\genfune^B|^2
+|\genfune^D|^2\right)\di\sigma\di t\leq \const\,,\qquad\forall\eps>0\,,
\end{equation}
with $\const$ independent of $\eps$, and
\begin{equation}\label{conv comp AB5}
\btsunfop([\genfune])\wto [\genfun]\,,\qquad \hbox{weakly in } L^2(\Om_T\times\Permemb)\,,
\end{equation}
where $[\genfune]=\genfune^B-\genfune^D$ and $[\genfun]=\genfun^B-\genfun^D$ and
we have denoted by $\genfune^D,\genfune^B$ the trace on $\Memb$ of $\genfune$ from $\Omint$ and $\Omout$, respectively.
Moreover, on $\Permemb$, we have also identified $\genfun^D,\genfun^B$ with their traces.
\end{prop}

\begin{proof}
The convergences \eqref{conv comp AB1}--\eqref{conv comp AB4} follow
by \cite[Theorem 2.13]{Cioranescu:Damlamian:Donato:Griso:Zaki:2012}
applied in $\Omint$ and $\Omout$, separately.
Inequality \eqref{eq:a31} is a consequence of the standard trace inequality together with
a rescaling argument. Finally, \eqref{conv comp AB5} follows from the fact that
\begin{equation}\label{eq:a2}
\btsunfop(\genfune^D)\wto \genfun^D\,,
\ \ \btsunfop(\genfune^B)\wto \genfun^B\,,
\qquad \hbox{weakly in } L^2(\Om_T\times\Permemb).
\end{equation}
Indeed, by \eqref{eq:a31}, we obtain that there exists $W\in L^2(\Om_T\times\Permemb)$ such that, up to a subsequence,
$\btsunfop(\genfune^D)\wto W$ weakly in $L^2(\Om_T\times\Permemb)$. Moreover, by Gauss-Green Theorem
and \eqref{conv comp AB1}--\eqref{conv comp AB4}, recalling that $\nabla_y\unfop(\genfune)=\eps\unfop(\nabla\genfune) $, we get
\begin{multline*}
\int_{\Om_T}\!\!\int_{\Permemb} W \varphi\cdot\nu_i\di \sigma \di x\di t\leftarrow
\int_{\Om_T}\!\!\int_{\Memb}\btsunfop(\genfune^D)\varphi\nu_i\di \sigma\di x\di t
=\int_{\Om_T}\!\int_{\Perint}\frac{\partial}{\partial y_i}\big(\unfop(\genfune)\varphi\big)\di y\di x\di t
\\
=\eps\int_{\Om_T}\!\int_{\Perint}\unfop(\partial_i\genfune)\varphi\di y\di x\di t
+\int_{\Om_T}\!\int_{\Perint}\unfop(\genfune)\frac{\partial\varphi}{\partial y_i}\di y\di x\di t
\\
\to \int_{\Om_T}\!\int_{\Perint}\genfun^D\frac{\partial\varphi}{\partial y_i}\di y\di x\di t
=\int_{\Om_T}\!\int_{\Perint}\frac{\partial}{\partial y_i}(\genfun^D\varphi)\di y\di x\di t
=\int_{\Om_T}\!\!\int_{\Permemb}\genfun^D\varphi\cdot \nu_i\di \sigma\di x\di t\,,
\end{multline*}
for any $\varphi \in {\CC}^1(\overline{\Om_T\times Y})$ with $supp(\varphi)\subset\subset \Om_T\times Y$ and
for $i=1,\dots,N$, where $\nu=(\nu_1,\dots,\nu_N)$ is the unit normal vector pointing into $\Perout$. This implies that
$W=\genfun^D$ on $\Permemb$. Clearly, the same procedure can be applied to $\genfune^B$ and $\genfun^B$.
\end{proof}

\begin{remark}\label{r:r2}
Notice that in the connected/disconnected geometry, the result stated in Proposition \ref{p:comp AB}
is still true, up to the fact that, now, $\genfun^D$ belongs only
to the space $ L^2(\Om_T)$ and, consequently, \eqref{conv comp AB3} is replaced by
\begin{equation}\label{eq:a13}
\unfop(\chi_{\Omint}\nabla\genfune)\wto \chi_{\Perint}\left(\nabla\genfun^B+\nabla_y\hat\genfun^D\right),
  \qquad \hbox{weakly in } L^2(\Om_T\times Y).
\end{equation}
Indeed, by \cite[Theorem 4.3]{Donato:Nguyen:2015}, we obtain
$$
\unfop(\chi_{\Omint}\nabla\genfune)\wto \chi_{\Perint}\nabla_y{\hvi}^D,
   \qquad\hbox{weakly in } L^2(\Om_T\times Y),
$$
for a suitable ${\hvi}^D\in L^2(\Om_T;H^1(\Perint))$, and, by
\cite[Remark 4.4]{Donato:Nguyen:2015}, we can identify $\nabla_y{\hvi}^D= \nabla \genfun^B+\nabla_y\hat\genfun^D$.
Moreover, it is worthwhile to remark that, in the connected/dis\-conneted geometry, the space $H^1(\Perint)$
coincides with $H^1_\#(\Perint)$.
\end{remark}

\section{Homogenization}
\label{s:homog}

In what follows, we extend $\ve$ to the whole of $\Om$ (still denoting the extension by $\ve$),
maintaining its $L^2(0,T;H^1(\Om))$-norm uniformly bounded, as it can be done following
\cite{Cioranescu:SaintJeanPaulin:1979,Tartar:1977}, in the connected/disconnected case,
and \cite{Acerbi:Chiado:DalMAso:Percivale:1992, Hopker:2016,Mabrouk:Hassan:2004}, in the connected/connected case.
We also extend $\we$ to the whole of $\Om$ (still keeping the notation $\we$), simply by taking $\we=0$ in $\Omint$.
\medskip

Our goal in this section is to describe the asymptotic behavior, as $\eps\to 0$, of the triplet $(\ve,\ue,\we)$ given
by the system \eqref{eq:PDEinc}--\eqref{eq:InitData3c}.
To this aim, we state the following compactness result.

\begin{lemma} \label{l:conv}
Suppose that $\capuno,\capdue,\dfinte,\dfoute,\dfdise,f_1,f_2,\overline v_0,\widetilde w_0\ s_{0\eps}$ satisfy
the assumptions stated in Subsection \ref{ss:position}. For every $\eps>0$, let $(\ve,\ue,\we)$ be the unique solution
of the system \eqref{eq:PDEinc}--\eqref{eq:InitData3c}, complemented with the gating problem
\eqref{eq:gating1}--\eqref{eq:gating2}. Then, up to a subsequence,
  still denoted by $\eps$, there exist  $v\in L^2(0,T;H^1_0(\Om))$, $\hat v\in L^2(\Om_T;H^1_\#(\Perout))$ with
   ${\mathcal{M}}_{\Perout} (\hat v)=0$, $u\in L^2(\Om_T)$, and $w\in L^2(\Om_T)$ such that
\begin{alignat}2
\label{conv1} &
\ve\wto v\qquad &\hbox{weakly in $L^2(0,T;H^1(\Om))$;}
\\
\label{conv1bis} &
\unfop(\ve)\wto v\qquad &\hbox{weakly in $L^2(\Om_T\times Y)$;}
\\
\label{conv2} &
\unfop(\chi_{\Omout}\nabla\ve)\wto \nabla v+\nabla_y \hat v\qquad &\hbox{weakly in $L^2(\Om_T\times\Perout)$.}
\end{alignat}
Moreover,
\begin{alignat}2
\label{conv7} &
\ve\to v\qquad &\hbox{strongly in $L^2(\Om_T)$;}
\\
\label{conv8} &
\ue\wto u\qquad &\hbox{weakly in $L^2(\Om_T)$;}
\\
\label{conv9}  &
\we\wto w\qquad &\hbox{weakly in $L^2(\Om_T)$.}
\end{alignat}
\end{lemma}

\begin{proof}
Assertions \eqref{conv1} and  \eqref{conv8} are direct consequence of the estimate
\eqref{eq:energy5}, while assertion \eqref{conv1bis}
follows by \cite[Theorem 2.11]{Amar:Andreucci:Gianni:Timofte:2017A}.
On the other hand, assertion \eqref{conv2} follows from Proposition \ref{t:smalleps_grad_weak_conv},
while \eqref{conv9} is a consequence of the
fact that $0\leq \we(x,t)\leq 1$ a.e. in $\Om_T$, uniformly with respect to $\eps$.
Finally, \eqref{conv7} follows from the next Proposition \ref{p:p1}.

To achieve the thesis, it remains to prove that the trace of $v$ on
$\partial\Om\times(0,T)$ is null.
As pointed out at the beginning of this section, we shall use different
extension operators in the two different geometries.
Therefore, in the connected/disconnected case, the null trace result is a direct consequence of the extension technique, while in the connected/connected one, it is due to \cite[Theorems 3.5 and 3.6]{Hopker:2016}, thanks to our geometrical assumptions.
\end{proof}

In the following proposition, we prove a strong convergence result for the difference $v_\eps$ of the intra and the extra-potentials
in the healthy zone. Such a result is not standard in homogenization theory, in particular when dealing with time-dependent function spaces.
\medskip

\begin{prop}\label{p:p1}
Under the assumptions of Lemma \ref{l:conv}, we have that
$\ve\to v$ strongly in $L^2(\Om_T)$.
\end{prop}

\begin{proof}
Following the ideas in \cite[Lemma 3.10]{Grandelius:2019},
let us consider the temporal translation $\ve_{\Delta t}(t)=\ve(t+\Delta t)$ and $\ue_{\Delta t}(t)=\ue(t+\Delta t)$
of $\ve$ and $\ue$, respectively. Clearly,
$\ve_{\Delta t}$ and $\ue_{\Delta t}$ satisfy the system \eqref{eq:PDEinc}--\eqref{eq:BoundDatac}
in $(0,T-\Delta t)$, with initial conditions $\ve_{\Delta t}(0)=\ve(\Delta t)$ and $\ue_{\Delta t}(0)=\ue(\Delta t)$.
We subtract the original equations from the corresponding ones satisfied by the temporal translated functions
and set $\hve(t)=\ve(t+\Delta t)-\ve(t)$ and $\hue(t)=\ue(t+\Delta t)-\ve(t)$ (the same notation will be
adopted for all the other quantities).
Thus, taking into account only equations \eqref{eq:PDEinc} and \eqref{eq:flux1c}
and using as test function $\testb(t)=-\int_t^{t+\Delta t}\ve(s)\di s$, we  obtain
\begin{multline}\label{eq:weak2}
\int_0^{T-\Delta t}\int_{\Omout} (\hve)^2\di x\di t=
\\
\int_0^{T-\Delta t}\int_{\Omout}\dfinte\nabla \hve\cdot\left(\int_t^{t+\Delta t}\nabla\ve(s)\di s\right)\di x\di t
+\int_0^{T-\Delta t}\int_{\Omout}\dfinte\nabla \hue\cdot\left(\int_t^{t+\Delta t}\nabla\ve(s)\di s\right)
\\
+\int_{\Omout} \big(\ve(T)-\ve(T-\Delta t)\big)\left(\int_{T-\Delta t}^{T}\ve(s)\di s\right)\di x
\\
-\int_{\Omout} \big(\ve(\Delta t)-\overline v_0\big)\left(\int_0^{\Delta t}\ve(s)\di s\right)\di x
\\
+\int_0^{T-\Delta t}\int_{\Omout} \big(I_{ion}(\ve_{\Delta t},\we_{\Delta t})-I_{ion}(\ve,\we)\big)
\left( \int_t^{t+\Delta t}\ve(s)\di s\right)\di x\di t
\\
-\int_0^{T-\Delta t}\int_{\Omout}\hat f_1\left(\int_t^{t+\Delta t}\ve(s)\di s\right)\di x\di t
=\sum_{k=1}^{6}I_k\,.
\end{multline}
Clearly, \eqref{eq:weak2} shall be complemented with the gating problems for $\we(t)$ and $\we(t+\Delta t)$.

Taking into account \eqref{eq:energy5} and using H\"{o}lder inequality,
we get
\begin{equation}\label{eq:a11}
\begin{aligned}
I_1 &=\int_0^{T-\Delta t}\int_{\Omout}\dfinte\nabla \hve\cdot\left(\int_t^{t+\Delta t}\nabla\ve(s)\di s\right)\di x\di t
\\
& \leq \const\left(\int_0^{T-\Delta t}\int_{\Omout}|\nabla \ve|^2\di x\di t\right)^{1/2}
\left\Vert\int_t^{t+\Delta t}\nabla\ve(s)\di s\right\Vert_{L^2(\Omout\times(0,T-\Delta t))}
\\
& \leq \const\Vert\nabla \ve\Vert^2_{L^2(\Omout_T)}\sqrt{\Delta t}
\leq \const \sqrt{\Delta t}\,.
\end{aligned}
\end{equation}
Similar computations lead to
\begin{equation*}
\begin{aligned}
& I_2\leq \const \Vert \nabla \ue\Vert_{L^2(\Omout_T)} \Vert \nabla \ve\Vert_{L^2(\Omout_T)}\sqrt{\Delta t}
\leq \const \sqrt{\Delta t}\,,
\\
& I_3\leq\const \sqrt{\Delta t}\sup_{t\in(0,T)} \int_{\Omout}(\ve)^2(x,t)\di x\leq\const\sqrt{\Delta t}\,,
\\
& I_4\leq\const \sqrt{\Delta t}
\sup_{t\in(0,T)} \left( \int_{\Omout}(\ve)^2(x,t)\di x
+\Vert \overline v_0\Vert_{L^2(\Om)}\left(\int_{\Omout}(\ve)^2(x,t)\di x\right)^{1/2}\right)\leq\const\sqrt{\Delta t}\,,
\end{aligned}
\end{equation*}
\begin{equation*}
\begin{aligned}
& I_6\leq \const \Vert \nabla \hat f_1\Vert_{L^2(\Omout_T)}\Vert \nabla \ve\Vert_{L^2(\Omout_T)} \sqrt{\Delta t}
\leq \const \sqrt{\Delta t}\,.
\end{aligned}
\end{equation*}
Finally,
\begin{multline*}
I_5\leq \const\Vert I_{ion}(\ve_{\Delta t},\we_{\Delta t})-I_{ion}(\ve,\we)\Vert_{L^2(\Om\times(T-\Delta t))}
\Vert \ve\Vert_{L^2(\Om_T)}\sqrt{\Delta t}
\\
\leq \const \Vert \hve\Vert_{L^2(\Om\times(T-\Delta t))}\Vert \ve\Vert_{L^2(\Om_T)}\sqrt{\Delta t}
\leq \const\Vert \ve\Vert^2_{L^2(\Om_T)}\sqrt{\Delta t}\,,
\end{multline*}
where, in the second inequality, we used \eqref{eq:ion1}.
Collecting all the previous estimates, from \eqref{eq:weak2} we obtain
\begin{equation}\label{eq:a12}
\int_0^{T-\Delta t}\int_{\Omout} |\ve(x,t+\Delta t)-\ve(x,t)|^2\di x\di t\leq \const \sqrt{\Delta t}\,.
\end{equation}
Therefore, taking into account the energy estimate \eqref{eq:energy5}, by \eqref{eq:a12} we can infer that
$\ve\to v$ strongly in $L^2(\Om_T)$.
\end{proof}

\begin{lemma}\label{l:conv1}
Under the assumptions of Lemma \ref{l:conv}, we have that, up to a subsequence,
still denoted by $\eps$, there exist  $u^B\in L^2(0,T;H^1_0(\Om))$, $u^D\in L^2(\Om_T)$ and
$\hat u=(\hat u^B,\hat u^D)\in L^2(\Om_T;\X^1_\#(Y))$
with ${\mathcal{M}}_Y (\hat u)=0$, such that
\begin{equation}\label{conv4}
\unfop(\chi_{\Omout}\nabla\ue)\wto \nabla u^B+\nabla_y \hat u^B,\qquad\hbox{weakly in $L^2(\Om_T\times\Perout)$.}
\end{equation}
Moreover,
\begin{itemize}
\item in the connected/connected case, $u^D\in L^2(0,T;H^1_0(\Om))$ and
\begin{equation}\label{conv5}
\unfop(\chi_{\Omint}\nabla\ue)\wto \nabla u^D+\nabla_y \hat u^D,\qquad \hbox{weakly in $L^2(\Om_T\times\Perint)$;}
\end{equation}
\item in the
connected/disconnected case,
\begin{equation}\label{conv5bisse}
\unfop(\chi_{\Omint}\nabla\ue)\wto \nabla u^B+\nabla_y \hat u^D,\qquad \hbox{weakly in $L^2(\Om_T\times\Perint)$;}
\end{equation}
\item for $\espo>-1$ and in both geometries, we have
\begin{equation}\label{conv6}
\btsunfop([\ue])\to 0\qquad \hbox{weakly in $L^2(\Om_T\times\Permemb)$,}
\end{equation}
so that $u^B=u^D=:u\in L^2(0,T;H^1_0(\Om))$.
\end{itemize}
\end{lemma}

\begin{proof}
By \eqref{eq:energy5}, it follows that Proposition \ref{p:comp AB} and Remark \ref{r:r2} hold.
Therefore, assertions \eqref{conv4}, \eqref{conv5} and \eqref{conv5bisse} are direct consequence of \eqref{conv comp AB3},
\eqref{conv comp AB4} and \eqref{eq:a13}.

Moreover, in the connected/connected case,
the traces of $u^B$ and $u^D$ are zero on $\partial\Om\times(0,T)$. Indeed, we can identify $\ue\mid_{\Omout}=\ubdue$
and $\ue\mid_{\Omint}=\ud$ (as already done for $\ve$ and
$\we$) with their extensions and then apply \cite[Theorems 3.5 and 3.6]{Hopker:2016}.
In the connected/disconnected case, we have only to identify
$\ue\mid_{\Omout}=\ubdue$ with its extension from outside, so that $\ubdue\in L^2(0,T;H^1_0(\Om))$. Moreover, by the energy estimate \eqref{eq:energy5} and the properties of the extension, we obtain also that $\ubdue \wto u^B$ weakly in $L^2(0,T;H^1_0(\Om))$.

Finally, in order to prove \eqref{conv6}, recalling that $\espo>-1$ and taking into account
\eqref{eq:energy5}, we obtain that
$$
\frac{1}{\eps^\espo}\int_{\Memb_T}[\ue]^2(t)\di\sigma\di t\leq \const\,,
$$
with $\const$ independent of $\eps$. Therefore, it follows that (see \cite[Inequality (34)]{Amar:Gianni:2016A})
$$
\int_{\Om_T\times\Permemb}\btsunfop([\ue]^2)\di\sigma\di x\di t\leq \eps \int_{\Memb_T}[\ue]^2\di\sigma\di t
\leq \const\eps^{1+\espo}\to 0\,.
$$
Hence, $ \btsunfop([\ue])\to 0$ strongly in $L^2(\Om_T\times  \Permemb)$.
Taking into account that, by \eqref{conv comp AB5}, $\btsunfop([\ue])\wto [u]$ weakly in $L^2(\Om_T\times  \Permemb)$,
we get $[u]=0$, i.e. $u\in L^2(0,T;H^1_0(\Om))$.
\end{proof}

We now have to distinguish the different scalings.

\subsection{The scaling $\espo=1$}\label{ss:scaling1}

In addition to what stated in Lemmas \ref{l:conv} and \ref{l:conv1}, we can also state the following result.

\begin{lemma}\label{l:conv2}
Under the assumptions of Lemma \ref{l:conv}, we have that, up to a subsequence, still denoted by $\eps$,
\begin{equation}
\label{conv7bis}
\btsunfop(\eps^{-1}[\ue])\wto [\hat u]\qquad \hbox{weakly in $L^2(\Om_T\times\Permemb)$.}
\end{equation}
\end{lemma}

\begin{proof}
Assertion \eqref{conv7bis} is a consequence of the result in \cite[Theorem 3]{Nguyen:2015}
and \cite[Theorem 4.3 and Remark 4.4]{Donato:Nguyen:2015}, once we redefine
$$
\hat u^B=\hat u^1+\xi_{\Gamma}-m\,,
\qquad
\hat u^D=\hat u^2-m\,,
$$
where $\hat u^1\in L^2(\Om_T;H^1_\#(\Perout))$, $\hat u^2\in L^2(\Om_T;H^1_\#(\Perint))$,
$\xi_\Gamma, m\in L^2(\Om_T)$, with $\hat u^1, \hat u^2,\xi_\Gamma$ given in \cite[Theorem 3]{Nguyen:2015} and
$m=\mathcal{M}_{\Perout}(\hat u^1) +|\Perout|\xi_\Gamma +\mathcal{M}_{\Perint}(\hat u^2)$.
\end{proof}

\begin{thm}\label{t:hom1}
Let $\capuno,\capdue,\dfinte,\dfoute,\dfdise,f_1,f_2,\overline v_0,\win$ and $s_{0\eps}$ be as
in Subsection \ref{ss:position}.
Assume also that $\btsunfop(\eps^{-1}s_{0\eps})\wto s_1$ weakly in $L^2(\Om\times\Permemb)$.
For every $\eps>0$, let $(\ve,\ue,\we)$ be the unique solution
of the system \eqref{eq:PDEinc}--\eqref{eq:InitData3c}, complemented with the gating problem
\eqref{eq:gating1}-\eqref{eq:gating2}. Then,
there exist  $v,u\in L^2(0,T;H^1_0(\Om))$, $\hat v\in L^2(\Om_T;H^1_\#(\Perout))$ with
   ${\mathcal{M}}_{\Perout} (\hat v)=0$, $\hat u\in L^2(\Om_T;\X^1_\#(Y))$ with ${\mathcal{M}}_Y (\hat u)=0$,
  and $w\in L^2(\Om_T)$, such that $\ve\wto v$, $\ue \wto u$, $\we\wto w$ in the sense of Lemmas \ref{l:conv}, \ref{l:conv1}
  and \ref{l:conv2}.
  Moreover, $v,\hat v,u,\hat u,w$ are the unique solutions of the two-scale homogenized system given by
\begin{alignat}2
\nonumber &
|\Perout|v_t-\Div\left(\int_{\Perout}\dfint(\nabla (v+u)+\nabla_y(\hat v+\hat u^B))\di y\right)&
\\
\label{eq:PDEhom1} & \qquad\qquad\qquad\qquad\qquad
+|\Perout|I_{ion}(v,w)= |\Perout| f_1,\ &\text{in $\Om_T$;}
\\
\nonumber &
-\Div\!\left(\int_{\Perout}\big\{(\dfint+\dfout)(\nabla u\!+\!\nabla_y\hat u^B)+\dfint(\nabla v\!+\!\nabla_y\hat v)\!\big\}\di y\right)
&
\\
\label{eq:PDEhom2}& \qquad\qquad
-\!\Div\left(\int_{\Perint} \dfdis(\nabla u\!+\!\nabla \hat u^D)\di y\right)= |\Perout|(f_1-f_2),\ &\text{in $\Om_T$;}
\\
\end{alignat}
\begin{alignat}2
\label{eq:PDEmicrohom1} &
-\Div_y(\dfint\nabla (v+u)+\dfint\nabla_y(\hat v+\hat u^B))=0,\qquad
&\text{in $\Om_T\times \Perout$;}
\\
\label{eq:PDEmicrohom2} &
-\Div_y(\dfout(\nabla u+\nabla_y\hat u^B))=0,\
&\text{in $\Om_T\times \Perout$;}
\\
\label{eq:PDEmicrohom3} &
-\Div_y(\dfdis(\nabla u+\nabla_y\hat u^D))=0,\
&\text{in $\Om_T\times \Perint$;}
\\
\label{eq:Fluxhom1} &
\dfint\nabla (v+u)+\dfint\nabla_y(\hat v+\hat u^B)\cdot\nu =0,\
&\text{on $\Om_T\times \Permemb$;}
\\
\label{eq:Fluxhom2} &
[\dfboth(\nabla u+\nabla_y\hat u)\cdot\nu] =0,\
&\text{on $\Om_T\times \Permemb$;}
\\
\label{eq:Circuithom1} &
\capuno[\hat u]_t+\capdue[\hat u]=\dfout(\nabla u+\nabla_y\hat u^B)\cdot\nu ,\
&\text{on $\Om_T\times \Permemb$;}
\\
\label{eq:InitDatahom1} &
v(x,0)=\overline v_0,&\text{in $\Om$;}
\\
\label{eq:InitDatahom2} &
[\hat u](x,y,0)=s_1,&\text{in $\Om\times \Permemb$;}
\\
\label{eq:BoundDatahom} &
v,u=0,&\text{on $\partial\Om$,}
\end{alignat}
and
\begin{alignat}2
\label{eq:gating3} &
\partial_t w+g(v,w)=0\,,\qquad &\text{in $\Om_T$;}
\\
\label{eq:gating4} &
w(x,0)=\win(x)\,,\qquad &\text{in $\Om$.}
\end{alignat}
Here, $\dfboth$ is the matrix which coincides with $\dfout$ in $\Perout$ and with $\dfdis$ in $\Perint$.
\end{thm}

\begin{proof}
In the weak formulation \eqref{eq:weak5}, let us take, as test functions,
\begin{equation*}
\testb=\phi_B(x,t)+\eps\psi_B(x,t,x/\eps)
\qquad\hbox{and}\qquad
\testd=\phi_D(x,t)+\eps\psi_D(x,t,x/\eps)\,,
\end{equation*}
where $\phi_B,\phi_D\in \CC^1(\overline\Om_T)$, with compact support in $\Om$, for every $t\in[0,T]$,
and such that $\phi_B(x,T)=\phi_D(x,T)=0$, for every $x\in\overline\Om$, $\psi_B\in \CC^1(\overline\Om_T;\CC^1_\#(\overline\Perout))$,
with compact support in $\Om$, for every $(t,y)\in[0,T]\times\overline Y$, and such that $\psi_B(x,T,y)=0$, for every $(x,y)\in\overline\Om\times
\overline Y$, $\psi_D\in \CC^1(\Om_T;\X^1_\#(Y))$, with compact support in $\overline\Om$, for every $(t,y)\in[0,T]\times\overline Y$,
and such that $[\psi_D(x,T,y)]=0$, for every $(x,y)\in\overline\Om\times \overline Y$.
Then, we obtain
\begin{multline*}
-\int_{\Omout_T} \ve(\partial_t\phi_B+\eps\partial_t\psi_B)\di x\di t+\int_{\Omout_T}\dfinte\nabla \ve\cdot(\nabla \phi_B+\eps\nabla_x\psi_B+\nabla_y\psi_B)\di x\di t
\\
+\int_{\Omout_T}\dfinte\nabla \ue\cdot(\nabla \phi_B+\eps\nabla_x\psi_B+\nabla_y\psi_B)\di x\di t
+\int_{\Omout_T} I_{ion}(\ve,\we)(\phi_B+\eps\psi_B)\di x\di t
\\
+\int_{\Omout_T}(\dfinte+\dfoute)\nabla \ue\cdot (\nabla \phi_D+\eps\nabla_x\psi_D^1+\nabla_y\psi_D^1)\di x\di t
\\
+\!\!\int_{\Omout_T}\!\dfinte\nabla \ve\cdot(\nabla \phi_D+\eps\nabla_x\psi_D^1+\nabla_y\psi_D^1)\di x\di t
+\!\!\int_{\Omint_T}\!\dfdise\nabla  \ue\cdot (\nabla \phi_D+\eps\nabla_x\psi_D^2+\nabla_y\psi_D^2)\di x\di t
\\
-{\capuno}{\eps}\int_{\Memb_T}\frac{[\ue]}{\eps}\partial_t[\psi_D]\di\sigma\di t
+{\capdue}\eps\int_{\Memb_T}\frac{[\ue]}{\eps}[\psi_D]\di\sigma\di t
\end{multline*}
\begin{multline}\label{eq:weak1}
=\int_{\Omout_T}f_1(\phi_B+\eps\psi_B)\di x\di t+\int_{\Omout_T}(f_1-f_2)(\phi_D+\eps\psi_D^1)\di x\di t
\\
+\int_{\Omout} \overline v_0(\phi_B(0)+\eps\psi_B(0))\di x
+{\capuno\eps}\int_{\Memb}\frac{s_{0\eps}}{\eps}[\psi_D](0)\di\sigma\,.
\end{multline}
Unfolding and passing to the limit, we arrive at
\begin{multline}\label{eq:a9}
-|\Perout|\int_{\Om_T}v \partial_t\phi_B\di x\di t+\int_{\Om_T}\int_{\Perout}\dfint(\nabla v+\nabla_y\hat v)\cdot
(\nabla\phi_B+\nabla_y\psi_B)\di y\di x\di t
\\
+\int_{\Om_T}\int_{\Perout}\dfint(\nabla u+\nabla_y\hat u^B)\cdot(\nabla\phi_B+\nabla_y\psi_B)\di y \di x\di t
+|\Perout|\int_{\Om_T}I_{ion}(v,w)\phi_B\di x\di t
\\
+\int_{\Om_T}\int_{\Perout}(\dfint+\dfout)(\nabla u+\nabla_y\hat u^B)\cdot(\nabla\phi_D+\nabla_y\psi_D^1)\di y\di x\di t
\\
+\int_{\Om_T}\int_{\Perout}\dfint(\nabla v+\nabla_y\hat v)\cdot(\nabla\phi_D+\nabla_y\psi_D^1)\di y\di x\di t
\\
+\int_{\Om_T}\int_{\Perint}\dfdis(\nabla u+\nabla_y\hat u^D)\cdot(\nabla \phi_D+\nabla_y\psi_D^2)\di y\di x \di t
\\
-\capuno\int_{\Om_T}\int_{\Permemb}[\hat u]\partial_t[\psi_D]\di\sigma\di x \di t
+\capdue\int_{\Om_T}\int_{\Permemb}[\hat u][\psi_D]\di\sigma\di x\di t
\\
= |\Perout|\int_{\Om_T}f_1\phi_B\di x\di t+|\Perout|\int_{\Om_T}(f_1-f_2)\phi_D\di x\di t
\\
+|\Perout|\int_{\Om}\overline v_0\phi_B(0)\di x+\capuno\int_\Om\int_\Permemb s_1[\psi_D](0)\di\sigma\di x\,,
\end{multline}
where we have used Lemmas \ref{l:conv}, \ref{l:conv1} and \ref{l:conv2}.
In order to get the strong formulation \eqref{eq:PDEhom1}--\eqref{eq:BoundDatahom}, we localize \eqref{eq:a9}, taking first
$\psi_B=\phi_D=\psi_D^1=\psi_D^2=0$ and then $\phi_B=\psi_B=\psi_D^1=\psi_D^2=0$, so that we arrive at \eqref{eq:PDEhom1},
\eqref{eq:PDEhom2} and \eqref{eq:InitDatahom1}. Moreover, we take $\phi_B=\phi_D=\psi_D^1=\psi_D^2=0$, which gives
\eqref{eq:PDEmicrohom1} and \eqref{eq:Fluxhom1}.
In the next step, we take first $\phi_B=\phi_D=\psi_B=\psi_D^2=0$ and $\psi_D^1$ with compact support in $\Perout$ and
then $\phi_B=\phi_D=\psi_B=\psi_D^1=0$ and $\psi_D^2$ with compact support in $\Perint$, in order to obtain
\begin{equation}\label{eq:PDEmicrohom21}
-\Div_y((\dfint+\dfout)(\nabla u+\nabla_y\hat u^B))-\Div_y(\dfint(\nabla v+\nabla\hat v))=0,\qquad
\text{in $\Om_T\times \Perout$;}
\end{equation}
and \eqref{eq:PDEmicrohom3}, respectively. Clearly, subtracting \eqref{eq:PDEmicrohom1} from \eqref{eq:PDEmicrohom21}, we get also \eqref{eq:PDEmicrohom2}.
In the last step, we let $\phi_B=\phi_D=\psi_B=0$ and we take advantage of the equations previously found, obtaining
\begin{multline*}
\capuno\int_\Om\int_\Permemb s_1[\psi_D](0)\di\sigma\di x=
-\int_{\Om_T}\int_{\Permemb}(\dfint+\dfout)(\nabla u+\nabla_y\hat u^B)\cdot \nu\psi_D^1\di \sigma\di x\di t
\\
-\int_{\Om_T}\int_{\Permemb}\dfint(\nabla v+\nabla_y\hat v)\cdot\nu\psi_D^1\di \sigma\di x\di t
+\int_{\Om_T}\int_{\Permemb}\dfdis(\nabla u+\nabla_y\hat u^D)\cdot\nu\psi_D^2\di \sigma\di x \di t
\\
-\capuno\int_{\Om_T}\int_{\Permemb}[\hat u]\partial_t[\psi_D]\di\sigma\di x \di t
+\capdue\int_{\Om_T}\int_{\Permemb}[\hat u][\psi_D]\di\sigma\di x\di t=
\\
-\int_{\Om_T}\int_{\Permemb}\dfout(\nabla u+\nabla_y\hat u^B)\cdot \nu\psi_D^1\di \sigma\di x\di t
+\int_{\Om_T}\int_{\Permemb}\dfdis(\nabla u+\nabla_y\hat u^D)\cdot\nu\psi_D^2\di \sigma\di x \di t
\\
-\capuno\int_{\Om_T}\int_{\Permemb}[\hat u]\partial_t[\psi_D]\di\sigma\di x \di t
+\capdue\int_{\Om_t}\int_{\Permemb}[\hat u][\psi_D]\di\sigma\di x\di t=
\\
-\int_{\Om_T}\int_{\Permemb}[\dfboth(\nabla u+\nabla_y\hat u)\cdot \nu]\psi_D^2\di \sigma\di x\di t
-\int_{\Om_T}\int_{\Permemb}\dfout(\nabla u+\nabla_y\hat u^B)\cdot\nu[\psi_D]\di \sigma\di x \di t
\\
-\capuno\int_{\Om_T}\int_{\Permemb}[\hat u]\partial_t[\psi_D]\di\sigma\di x \di t
+\capdue\int_{\Om_T}\int_{\Permemb}[\hat u][\psi_D]\di\sigma\di x\di t\,,
\end{multline*}
where, in the second equality, we have taken into account \eqref{eq:Fluxhom1}. Therefore, if we take $[\psi]=0$,
it follows \eqref{eq:Fluxhom2}, while, when $[\psi]\not=0$, we get \eqref{eq:Circuithom1} and \eqref{eq:InitDatahom2}.
The boundary condition \eqref{eq:BoundDatahom} is a direct consequence of the fact that $v, u\in L^2(0,T;H^1_0(\Om))$.

Finally, the limit gating problem \eqref{eq:gating3}--\eqref{eq:gating4} follows from
\eqref{conv7}, \eqref{conv9}  and \eqref{eq:gating3bis}, once we pass to the limit in \eqref{eq:gating5}, written
for $p^\eps=\ubuno-\ubdue$, similarly as done in \cite[Proposition 4.7]{Collin:Imperiale:2018} and \cite[Section 5.3]{CDP2} (see, also,
\cite[Lemma 2.5]{Graf:Peter:2014}).

In order to conclude the proof, it remains to guarantee that the two-scale homogenized system \eqref{eq:PDEhom1}--\eqref{eq:BoundDatahom}
admits a unique solution, but this is a direct consequence of the linearity of the system jointly with the Lipschitz continuity of $I_{ion}$.
Therefore, the whole sequence, and not only a subsequence, converges.
\end{proof}

\begin{thm}\label{t:t1}
The two-scale system \eqref{eq:PDEhom1}--\eqref{eq:BoundDatahom} can be rewritten as the single-scale degenerate parabolic system
given by
\begin{equation}\label{eq:a15}
\begin{aligned}
\partial_t v-\Div\left(\Ahomuno\nabla (v+u)\right)+I_{ion}(v,w)=f_1\,,\qquad & \hbox{in $\Om_T$;}
\\
-\Div\left(\Ahomuno\nabla (v+u)\right)-\Div\left(\Ahomdue\nabla u+\int_0^t \Ahomtre(t-\tau)\nabla u(\tau)\di \tau\right)
&
\\
=\mathcal{F}+(f_1-f_2)\,,\qquad & \hbox{in $\Om_T$,}
\end{aligned}
\end{equation}
complemented with the initial and the boundary conditions \eqref{eq:InitDatahom1}, \eqref{eq:BoundDatahom} and the gating problem
\eqref{eq:gating3}--\eqref{eq:gating4}. Here, the matrices $\Ahomuno,\,\Ahomdue,\Ahomtre$ are defined in \eqref{eq:matrix2}
and $\mathcal{F}$ is defined in \eqref{eq:a16}. Moreover, the matrices $\Ahomuno$ and $\Ahomdue$ are symmetric and positive definite
and $\Ahomtre$ is symmetric.
\end{thm}

\begin{proof}
Taking into account \eqref{eq:PDEmicrohom1} and \eqref{eq:Fluxhom1}, we can factorize
\begin{equation}\label{eq:factor1}
(\hat v+\hat u^B)(x,y,t)=-\zeta(y)\cdot\nabla(v+u)(x,t)\,,
\end{equation}
where the cell functions $\zeta=(\zeta^1,\dots,\zeta^N)$, with $\zeta^j\in H^1_{\#}(\Perout)$ and $\mathcal{M}_{\Perout}(\zeta^j)=0$,
are the solutions of the cell problem
\begin{equation}\label{eq:cell1}
\begin{aligned}
-\Div_y(\dfint\nabla_y(y^j-\zeta^j))=0\,,\qquad & \hbox{in $\Perout$;}
\\
\dfint\nabla_y(y^j-\zeta^j)\cdot\nu=0\,,\qquad & \hbox{on $\Permemb$.}
\end{aligned}
\end{equation}
Moreover, following \cite[Section 3]{Amar:Andreucci:Bisegna:Gianni:2004a} and taking into account \eqref{eq:PDEmicrohom2}--\eqref{eq:PDEmicrohom3}, \eqref{eq:Fluxhom2}--\eqref{eq:Circuithom1}
and \eqref{eq:InitDatahom2}, we can factorize
\begin{equation}\label{eq:factor2}
\hat u(x,t,y)=-\chi_0(y)\cdot \nabla u(x,t)-\int_0^t \chi_1(y,t-\tau)\cdot\nabla u(x,\tau)\di\tau+
{\mathcal T}(s_1(x,\cdot))(t,y)\,,
\end{equation}
where we need two families of cell functions $\chi_0=(\chi_0^1,\dots,\chi_0^N)$, with $\chi^j_0\in H^1_{\#}(Y)$ and
$\mathcal{M}_{Y}(\chi_0^j)=0$, and $(\chi_1^1,\dots,\chi_1^N)$, with $\chi_1^j\in \X^1_{\#}(Y)$ and $\mathcal{M}_{Y}(\chi_1^j)=0$.
More precisely, for $j=1,\dots, N$, $\chi_0^j$ satisfies the cell problem
\begin{equation}\label{eq:cell2}
\begin{aligned}
-\Div_y(\dfout\nabla_y(y^j-\chi_0^j))=0\,,\qquad & \hbox{in $\Perout$;}
\\
-\Div_y(\dfdis\nabla_y(y^j-\chi_0^j))=0\,,\qquad & \hbox{in $\Perint$;}
\\
[\dfboth\nabla_y(y^j-\chi_0^j)\cdot\nu]=0\,,\qquad & \hbox{on $\Permemb$,}
\end{aligned}
\end{equation}
which can be simply rewritten as
\begin{equation}\label{eq:a10}
-\Div_y(\dfboth\nabla_y(y^j-\chi_0^j))=0\,,\qquad\hbox{in $Y$.}
\end{equation}
In turn, for $j=1,\dots,N$, $\chi_1^j$ satisfies the cell problem
\begin{equation}\label{eq:cell3}
\begin{aligned}
-\Div_y(\dfout\nabla_y\chi_1^j)=0\,,\qquad & \hbox{in $\Perout\times(0,T)$;}
\\
-\Div_y(\dfdis\nabla_y\chi_1^j)=0\,,\qquad & \hbox{in $\Perint\times(0,T)$;}
\\
[\dfboth\nabla_y\chi_1^j\cdot\nu]=0\,,\qquad & \hbox{on $\Permemb\times(0,T)$;}
\\
\capuno\partial_t[\chi^j_1]+\capdue [\chi^j_1]=\dfout\nabla_y\chi_1^j\cdot\nu\,,\qquad & \hbox{on $\Permemb\times(0,T)$;}
\\
\capuno[\chi^j_1](0) =\dfout\nabla_y(\chi_0^j-y^j)\cdot\nu\,,\qquad & \hbox{on $\Permemb$.}
\end{aligned}
\end{equation}
Finally, ${\mathcal T}(s_1)\in L^2(\Om_T; \X^1_{\#}(Y))$, for a.e. $x\in\Om$, is defined as the solution of the problem
\begin{equation}\label{eq:a14}
\begin{aligned}
-\Div_y(\dfout\nabla_y{\mathcal T}(s_1))=0\,,\qquad & \hbox{in $\Perout\times(0,T)$;}
\\
-\Div_y(\dfdis\nabla_y{\mathcal T}(s_1))=0\,,\qquad & \hbox{in $\Perint\times(0,T)$;}
\\
[\dfboth\nabla_y{\mathcal T}(s_1)\cdot\nu]=0\,,\qquad & \hbox{on $\Permemb\times(0,T)$;}
\\
\capuno\partial_t[{\mathcal T}(s_1)]+\capdue [{\mathcal T}(s_1)]=\dfout\nabla_y{\mathcal T}(s_1)\cdot\nu\,,\qquad & \hbox{on $\Permemb\times(0,T)$;}
\\
\capuno[{\mathcal T}(s_1)](0) =s_1\,,\qquad & \hbox{on $\Permemb$,}
\end{aligned}
\end{equation}
with the additional condition $\mathcal{M}_Y({\mathcal T}(s_1))=0$, a.e. in $\Om_T$.

Notice that well-posedness for \eqref{eq:cell1} and \eqref{eq:a10} is a classical problem, while
systems \eqref{eq:cell3} and \eqref{eq:a14} admit a unique solution by \cite[Theorem 6 and Remark 7]{Amar:Andreucci:Bisegna:Gianni:2005}.

Inserting \eqref{eq:factor1} and \eqref{eq:factor2} in \eqref{eq:PDEhom1} and \eqref{eq:PDEhom2}, we get the single-scale
homogenized degenerate parabolic system \eqref{eq:a15},
where the matrices $\Ahomuno,\,\Ahomdue$ and $\Ahomtre$ are defined as
\begin{equation}\label{eq:matrix2}
\begin{aligned}
\Ahomuno & =\frac{1}{|\Perout|}\int_{\Perout}\dfint\nabla_y(y-\zeta)\di y =
\frac{1}{|\Perout|}\int_{\Perout} (\nabla_y(y-\zeta))^T\dfint\nabla_y(y-\zeta)\di y\,,
\\
\Ahomdue & =\frac{1}{|\Perout|}\left(\int_{\Perout}\dfout\nabla_y(y-\chi_0)\di y+\int_{\Perint}\dfdis\nabla_y(y-\chi_0)\di y\right)
\\
&=\frac{1}{|\Perout|}\int_{Y}\dfboth\nabla_y(y-\chi_0)\di y
=\frac{1}{|\Perout|}\int_{Y}(\nabla_y(y-\chi_0))^T\dfboth\nabla_y(y-\chi_0)\di y\,,
\\
\Ahomtre(t) & = -\frac{1}{|\Perout|}\int_{Y} \dfboth\nabla_y\chi_1(y,t)\di y
\,,
\end{aligned}
\end{equation}
and
\begin{equation}\label{eq:a16}
\mathcal{F}=\frac{1}{|\Perout|}\Div\left(\int_Y\dfboth\nabla_y\mathcal{T}(s_1)\di y\right)\,.
\end{equation}
Clearly, $\Ahomuno$ and $\Ahomdue$ are symmetric and their positive definiteness is a standard matter.
Regarding $\Ahomtre$, we notice that, in the case of $\dfboth$ constant in $\Perint$ and $\Perout$ (with possibly two different constants), using Gauss-Green
formula, it can be written as
$$
\Ahomtre(t)=\frac{1}{|\Perout|}\int_{\Permemb}[\dfboth\chi_1(y,t)]\otimes\nu\di \sigma\,,
$$
whose symmetry has been proved in \cite[Corollary 4.1]{Amar:Andreucci:Bisegna:Gianni:2004a}. However, still using the ideas in
\cite[Section 4]{Amar:Andreucci:Bisegna:Gianni:2004a}, we can prove that $\Ahomtre$ is symmetric also for a non piecewise constant
matrix $\dfboth$, satisfying \eqref{eq:matrix}.
Indeed, by \cite[Lemma 4.1]{Amar:Andreucci:Bisegna:Gianni:2004a} (applied to $s_1=\dfout \nabla_y(\chi_0^j-y^j)\cdot\nu$ and
$s_2=\dfout \nabla_y(\chi_0^h-y^h)\cdot\nu$), it follows
\begin{equation}\label{eq:a32}
\int_{\Permemb}[\chi_1^j](t)[\chi_1^h](0)\di\sigma = \int_{\Permemb}[\chi_1^j](0)[\chi_1^h](t)\di\sigma \,.
\end{equation}
Moreover, recalling the initial condition in \eqref{eq:cell3}, we also have
\begin{equation}\label{eq:a33}
\capuno\int_{\Permemb}[\chi_1^j](t)[\chi_1^h](0)\di\sigma =\int_{\Permemb}[\chi_1^j](t)\ \dfout \nabla_y(\chi_0^h-y^h)\cdot\nu\di\sigma\,.
\end{equation}
Now, let us take $\chi^j_1$ as test function for the cell equation \eqref{eq:a10} (written for $\chi_0^h$) and
$\chi^h_0$ as test function for the cell problem \eqref{eq:cell3}. We get
\begin{align*}
& \int_Y\dfboth\nabla_y(\chi^h_0-y^h)\nabla_y\chi^j_1\di y=-\int_{\Permemb}\dfout\nabla_y(\chi^h_0-y^h)\cdot\nu[\chi^j_1]\di\sigma\,,
\\
& \int_Y\dfboth\nabla_y\chi^j_1\nabla_y\chi^h_0\di y=0\,,
\end{align*}
which implies
\begin{multline*}
\Ahomtre_{hj}(t)=-\frac{1}{|\Perout|}\int_Y \dfboth {\rm\bold e}^h\nabla_y\chi^j_1(y,t)\di y
=-\frac{1}{|\Perout|}\int_{\Permemb}\dfout\nabla_y(\chi^h_0-y^h)\cdot\nu[\chi^j_1](t)\di\sigma
\\
=-\frac{\capuno}{|\Perout|}\int_{\Permemb}[\chi_1^j](t)[\chi_1^h](0)\di\sigma\,,
\end{multline*}
where, in the last equality, we have used \eqref{eq:a33}.
Reasoning as above, we arrive also to
\begin{equation*}
\Ahomtre_{jh}(t)=-\frac{1}{|\Perout|}\int_Y \dfboth {\rm\bold e}^j\nabla_y\chi^h_1(y,t)\di y
=-\frac{\capuno}{|\Perout|}\int_{\Permemb}[\chi_1^h](t)[\chi_1^j](0)\di\sigma\,.
\end{equation*}
Finally, the symmetry is proven taken into account \eqref{eq:a32}.
\end{proof}

\begin{remark}\label{r:r9}
Notice that the limit problem \eqref{eq:a15} leads to a bidomain model with memory effects.
Indeed, let us denote by $u^B$ and $u^D$, respectively, the limits of the functions $\ubuno$ and $\ud$,
appearing in the system \eqref{eq:PDEin}--\eqref{eq:InitData3}. Recalling that $\ve=\ubuno-\ubdue$,
$\ue\mid_{\Omout_{T}}=\ubdue$ and $\ue\mid_{\Omint_{T}}=\ud$, and taking into account \eqref{conv6},
we can replace $v=u^B-u^D$ and $u=u^D$ in \eqref{eq:a15}, thus obtaining
\begin{equation}\label{eq:a15modif}
\begin{aligned}
\partial_t (u^B-u^D)-\Div\left(\Ahomuno\nabla u^B\right)+I_{ion}(u^B-u^D,w)=f_1\,,\qquad & \hbox{in $\Om_T$;}
\\
-\Div\left(\Ahomuno\nabla u^B\right)-\Div\left(\Ahomdue\nabla u^D+\int_0^t \Ahomtre(t-\tau)\nabla u^D(\tau)\di \tau\right)
&
\\
=\mathcal{F}+(f_1-f_2)\,,\qquad & \hbox{in $\Om_T$.}
\end{aligned}
\end{equation}
\end{remark}

\begin{remark}\label{r:r12}
In the case $\espo >1$, by unfolding the last inequality in the energy estimate \eqref{eq:energy5},
we obtain
\begin{equation*}
\btsunfop(\eps^{-1}[\ue])\wto 0\qquad \hbox{strongly in $L^2(\Om_T\times\Permemb)$.}
\end{equation*}
In particular, this implies that there is no jump in the corrector $\hat u$, and
thus one can check that the limit problem is standard.
\end{remark}

\subsection{The scaling $\espo\in (-1,1)$}\label{ss:scaling2}

We recall that Lemmas \ref{l:conv} and \ref{l:conv1} are still in force.

\begin{thm}\label{t:hom2}
Assume that $\capuno,\capdue,\dfinte,\dfoute,\dfdise,f_1,f_2,\overline v_0,\win$ and $s_{0\eps}$ are as
in Subsection \ref{ss:position}.
For every $\eps>0$, let $(\ve,\ue,\we)$ be the unique solution
of the system \eqref{eq:PDEinc}--\eqref{eq:InitData3c}, complemented with the gating problem
\eqref{eq:gating1}--\eqref{eq:gating2}. Then,
there exist  $v,u\in L^2(0,T;H^1_0(\Om))$, $\hat v\in L^2(\Om_T;H^1_\#(\Perout))$ with
   ${\mathcal{M}}_{\Perout} (\hat v)=0$, $\hat u\in L^2(\Om_T;\X^1_\#(Y))$ with ${\mathcal{M}}_{\Perout} (\hat u^B)=0
   ={\mathcal{M}}_{\Perint} (\hat u^D)$,
  and $w\in L^2(\Om_T)$, such that $\ve\wto v$, $\ue \wto u$, $\we\wto w$ in the sense of Lemmas \ref{l:conv} and \ref{l:conv1}.
  Moreover, $v,\hat v,u,\hat u,w$ are the unique solutions of two-scale homogenized system given by
  \eqref{eq:PDEhom1}--\eqref{eq:Fluxhom1}, with \eqref{eq:Fluxhom2} and  \eqref{eq:Circuithom1}  replaced with
\begin{alignat}2
\label{eq:Fluxhom2a0} &
\dfout(\nabla u+\nabla_y\hat u^B)\cdot\nu =0,\
&\text{on $\Om_T\times \Permemb$;}
\\
\label{eq:Fluxhom3a0} &
\dfdis(\nabla u+\nabla_y\hat u^D)\cdot\nu =0,\
&\text{on $\Om_T\times \Permemb$,}
\end{alignat}
complemented with the intial-boundary conditions \eqref{eq:InitDatahom1}, \eqref{eq:BoundDatahom} and the gating problem
\eqref{eq:gating3}--\eqref{eq:gating4}.
\end{thm}

\begin{proof}
The proof can be carried out as in the case of Theorem \ref{t:hom1}. The main difference is that, now,
the last integral in \eqref{eq:weak1} is replaced by
$$
\frac{\capuno}{\eps^\espo}\, \eps\int_{\Memb}s_{0\eps}[\psi^D](0)\di\sigma
$$
and the fifth line of \eqref{eq:weak1} is replaced by
\begin{equation}\label{eq:a18}
-\frac{\capuno}{\eps^\espo}\,\eps\int_{\Memb}[\ue]\partial_t[\psi^D]\,\di\sigma
+\frac{\capdue}{\eps^\espo}\,\eps\int_{\Memb}[\ue][\psi^D]\,\di\sigma\,.
\end{equation}
However, taking into account that \eqref{eq:a20} can be rewritten in the form
$$
{\capuno}\int_{\Om\times\Permemb}\left(\unfop\left(\frac{s_{0\eps}}{ \eps^{ \frac{\espo+1}{2} }}\right)\right)^2\di x\di \sigma
=
{\capuno}{\eps}\int_{\Memb}\left(\frac{s_{0\eps}}{ \eps^{ \frac{\espo+1}{2} }}\right)^2\di \sigma
=
\frac{\capuno}{\eps^\espo}\int_{\Memb}s^2_{0\eps}\di\sigma\leq \const\,,
$$
it follows
\begin{multline}\label{eq:a17}
\frac{\capuno}{\eps^\espo}\, \eps\int_{\Memb}s_{0\eps}[\psi^D](0)\di\sigma
=\capuno\eps^{\frac{1-\espo}{2}}\int_{\Om\times\Permemb}\unfop\left( \frac{s_{0\eps}}{ \eps^{ \frac{\espo+1}{2} }}\right)
\unfop([\psi^D])(0)\di x\di\sigma
\leq\const\eps^{\frac{1-\espo}{2}}\to 0\,,
\end{multline}
and, thanks to \eqref{eq:energy5}, similar computations lead to the result that also the integrals in \eqref{eq:a18}
tend to zero, for $\eps\to 0$.
Hence, passing to the limit in \eqref{eq:weak1}, taking into account the previous facts and, finally, localizing, we get the thesis.
\end{proof}

\begin{thm}\label{t:t1a0}
The two-scale system \eqref{eq:PDEhom1}--\eqref{eq:Fluxhom1}, \eqref{eq:Fluxhom2a0} and \eqref{eq:Fluxhom3a0} can be rewritten as the single-scale degenerate parabolic system given by
\begin{equation}\label{eq:a15_a0}
\begin{aligned}
\partial_t v-\Div\left(\Ahomuno\nabla (v+u)\right)+I_{ion}(v,w)=f_1\,,\qquad & \hbox{in $\Om_T$;}
\\
-\Div\left(\Ahomuno\nabla (v+u)\right)-\Div\left(\wAhomdue\nabla u\right)
=f_1-f_2\,,\qquad & \hbox{in $\Om_T$,}
\end{aligned}
\end{equation}
complemented with the initial and the boundary conditions \eqref{eq:InitDatahom1}, \eqref{eq:BoundDatahom} and the gating problem
\eqref{eq:gating3}--\eqref{eq:gating4}, where the matrix $\Ahomuno$ is given in \eqref{eq:matrix2} and
$\wAhomdue=\Ahomdue_B+\Ahomdue_D$ is defined in \eqref{eq:matrix2_a0} and \eqref{eq:matrix3_a0}.
\end{thm}

\begin{proof}
As in the case $\espo=1$, taking into account \eqref{eq:PDEmicrohom1} and \eqref{eq:Fluxhom1}, we can factorize
$\hat v+\hat u^B$ as in \eqref{eq:factor1},
where the cell functions $\zeta=(\zeta^1,\dots,\zeta^N)$, with $\zeta^j\in H^1_{\#}(\Perout)$ and $\mathcal{M}_{\Perout}(\zeta^j)=0$,
are the solutions of the cell problem \eqref{eq:cell1}.
Moreover, taking into account \eqref{eq:PDEmicrohom2}, \eqref{eq:PDEmicrohom3}, \eqref{eq:Fluxhom2a0} and \eqref{eq:Fluxhom3a0},
we can factorize
\begin{equation}\label{eq:factor2_a0}
\hat u(x,t,y)=-\widehat\chi_0(y)\cdot \nabla u(x,t)\,,
\end{equation}
where the cell functions $\widehat\chi_0=(\widehat\chi_0^1,\dots,\widehat\chi_0^N)$, with $\widehat\chi^j_0
=(\widehat\chi^{B,j}_0,\widehat\chi^{D,j}_0)\in \X^1_{\#}(Y)$ and
$\mathcal{M}_{\Perout}(\widehat\chi_0^{B,j})=0=\mathcal{M}_{\Perint}(\widehat\chi_0^{D,j})$, for $j=1,\dots, N$, satisfy the cell problem
\begin{equation}\label{eq:cell2_a0}
\begin{aligned}
-\Div_y(\dfout\nabla_y(y^j-\widehat\chi_0^{B,j}))=0\,,\qquad & \hbox{in $\Perout$;}
\\
\dfout\nabla_y(y^j-\widehat\chi_0^{B,j})\cdot\nu=0\,,\qquad & \hbox{on $\Permemb$;}
\\
-\Div_y(\dfdis\nabla_y(y^j-\widehat\chi_0^{D,j}))=0\,,\qquad & \hbox{in $\Perint$;}
\\
\dfdis\nabla_y(y^j-\widehat\chi_0^{D,j})\cdot\nu=0\,,\qquad & \hbox{on $\Permemb$,}
\end{aligned}
\end{equation}
which are two independent Neuman problems.
Inserting these factorizations in \eqref{eq:PDEhom1} and \eqref{eq:PDEhom2}, we get the single-scale
homogenized degenerate parabolic system \eqref{eq:a15_a0},
where the matrix $\Ahomuno$ coincides with the one defined in \eqref{eq:matrix2}, while $\wAhomdue=\Ahomdue_B+\Ahomdue_D$ is given by
\begin{alignat}2
\label{eq:matrix2_a0}
& \Ahomdue_B\! \!=\frac{1}{|\Perout|}\int_{\Perout}\!\!\dfout\nabla_y(y-\widehat\chi^B_0)\di y
=\frac{1}{|\Perout|}\int_{\Perout}\!(\nabla_y(y-\widehat\chi^B_0))^T\dfout\nabla_y(y-\widehat\chi^B_0)\di y,
\\
\label{eq:matrix3_a0}
& \Ahomdue_D \!\!=\frac{1}{|\Perout|}\int_{\Perint}\!\!\dfdis\nabla_y(y-\widehat\chi^D_0)\di y
=\frac{1}{|\Perout|}\int_{\Perint}\!(\nabla_y(y-\widehat\chi^D_0))^T\dfdis\nabla_y(y-\widehat\chi^D_0)\di y.
\end{alignat}
Clearly, $\Ahomdue_B$ and $\Ahomdue_D$ are symmetric and their positive definiteness is a standard matter.
\end{proof}

\begin{remark}\label{r:r10}
As in Remark \ref{r:r9}, let us denote by $u^B$ and $u^D$, respectively, the limits of the functions $\ubuno$ and $\ud$,
appearing in the system \eqref{eq:PDEin}--\eqref{eq:InitData3}, so that,
replacing $v=u^B-u^D$ and $u=u^D$ in \eqref{eq:a15_a0}, we obtain
\begin{equation}\label{eq:a15_a0modif}
\begin{aligned}
\partial_t (u^B-u^D)-\Div\left(\Ahomuno\nabla u^B\right)+I_{ion}(u^B-u^D,w)=f_1\,,\qquad & \hbox{in $\Om_T$;}
\\
-\Div\left(\Ahomuno\nabla u^B\right)-\Div\left(\wAhomdue\nabla u^D\right)
=f_1-f_2\,,\qquad & \hbox{in $\Om_T$.}
\end{aligned}
\end{equation}
\end{remark}

\begin{remark}\label{r:r5}
In the connected/disconnected case, $\chi_0^{D,j}(y)=y^j$, up to an additive constant, so that $\Ahomdue_D=0$.
Therefore, $\Ahomdue=\Ahomdue_B$ and the limit problem is affected only by the physical properties of the phase $\Perout$.
\end{remark}

\subsection{The scaling $\espo=-1$}\label{ss:scaling3}

In addition to what stated in Lemmas \ref{l:conv} and \ref{l:conv1}, we can also state the following result.

\begin{lemma}\label{l:conv3}
Under the assumptions of Lemma \ref{l:conv}, we have that, up to a subsequence, still denoted by $\eps$,
\begin{equation}
\label{conv7tris}
\btsunfop([\ue])\wto [u]\qquad \hbox{weakly in $L^2(\Om_T\times\Permemb)$.}
\end{equation}
Here, with a little abuse of notation, $[u]=u^B-u^D$.
\end{lemma}

\begin{proof}
Assertion \eqref{conv7tris} is a consequence of \eqref{eq:energy5} and \eqref{conv comp AB5} in Proposition \ref{p:comp AB}.
\end{proof}

\begin{thm}\label{t:hom1a1}
Assume to be in the connected/connected geometry.
Let $\capuno,\capdue,\dfinte,$ $\dfoute,\dfdise,f_1,f_2,\overline v_0,\win$ and $s_{0\eps}$ be
as in Subsection \ref{ss:position} and assume
that $\btsunfop(s_{0\eps})\wto \overline s_1$ weakly in $L^2(\Om\times\Permemb)$.
For every $\eps>0$, let $(\ve,\ue,\we)$ be the unique solution
of the system \eqref{eq:PDEinc}--\eqref{eq:InitData3c}, complemented with the gating problem
\eqref{eq:gating1}--\eqref{eq:gating2}. Then,
there exist  $v,u^B,u^D\in L^2(0,T;H^1_0(\Om))$, $\hat v\in L^2(\Om_T;H^1_\#(\Perout))$ with
   ${\mathcal{M}}_{\Perout} (\hat v)=0$, $\hat u\in L^2(\Om_T;\X^1_\#(Y))$ with ${\mathcal{M}}_{\Perout} (\hat u^B)=0=
   {\mathcal{M}}_{\Perint} (\hat u^D)$,
  and $w\in L^2(\Om_T)$, such that $\ve\wto v$, $\ue\chi_{\Omout_T} \wto u^B$, $\ue\chi_{\Omint_T} \wto u^D$,
  $\we\wto w$ in the sense of Lemmas \ref{l:conv}, \ref{l:conv1}
  and \ref{l:conv2}.
  Moreover, $v,\hat v,u^B,u^D,\hat u,w$ are the unique solutions of two-scale homogenized system given by
\begin{alignat}2
\nonumber &
|\Perout|v_t-\Div\left(\int_{\Perout}\dfint(\nabla (v+u)+\nabla_y(\hat v+\hat u^B))\di y\right)&
\\
\label{eq:PDEhom1a1} & \qquad\qquad\qquad\qquad\qquad
+|\Perout|I_{ion}(v,w)= |\Perout| f_1,\ &\text{in $\Om_T$;}
\\
\nonumber &
\capuno|\Permemb|\partial_t[u]+\capdue|\Permemb|[u]-\Div\left(\int_{\Perout}(\dfint+\dfout)(\nabla u^B\!+\!\nabla_y\hat u^B)\di y\right)
&
\\
\label{eq:PDEhom2a1}& \qquad\qquad
-\Div\left(\int_{\Perout} \dfint(\nabla v^B\!+\!\nabla_y \hat v^B)\di y\right)= |\Perout|(f_1-f_2),\ &\text{in $\Om_T$;}
\\
\label{eq:PDEhom3a1}& \qquad\qquad
\capuno|\Permemb|\partial_t[u]+\capdue|\Permemb|[u]+\Div\left(\int_{\Perint} \dfdis(\nabla u^D\!+\!\nabla_y \hat u^D)\di y\right)=0,\quad &\text{in $\Om_T$;}
\end{alignat}
\begin{alignat}2
\label{eq:PDEmicrohom1a1} &
-\Div_y(\dfint\nabla (v+u^B)+\dfint\nabla_y(\hat v+\hat u^B))=0,\qquad\qquad
&\text{in $\Om_T\times \Perout$;}
\\
\label{eq:PDEmicrohom2a1} &
-\Div_y(\dfout(\nabla u^B+\nabla_y\hat u^B))=0,\
&\text{in $\Om_T\times \Perout$;}
\\
\label{eq:PDEmicrohom3a1} &
-\Div_y(\dfdis(\nabla u^D+\nabla_y\hat u^D))=0,\
&\text{in $\Om_T\times \Perint$;}
\\
\label{eq:Fluxhom1a1} &
\dfint\nabla (v+u^B)+\dfint\nabla_y(\hat v+\hat u^B)\cdot\nu =0,\
&\text{on $\Om_T\times \Permemb$;}
\\
\label{eq:Fluxhom2a1} &
\dfout(\nabla u^B+\nabla_y\hat u^B)\cdot\nu =0,\
&\text{on $\Om_T\times \Permemb$;}
\\
\label{eq:Fluxhom3a1} &
\dfdis(\nabla u^D+\nabla_y\hat u^D)\cdot\nu =0,\
&\text{on $\Om_T\times \Permemb$;}
\\
\label{eq:InitDatahom1a1} &
v(x,0)=\overline v_0,&\text{in $\Om$;}
\\
\label{eq:InitDatahom2a1} &
[u](x,0)=\frac{1}{|\Permemb|}\int_{\Permemb} \overline s_1\di\sigma,&\text{in $\Om$;}
\\
\label{eq:BoundDatahoma1} &
v,u^B,u^D=0,&\text{on $\partial\Om$,}
\end{alignat}
complemented with the gating problem \eqref{eq:gating3}-\eqref{eq:gating4}.
\end{thm}

\begin{remark}\label{r:r7}
In the connected/disconnected case, $u^D\in L^2(\Om_T)$ and equations \eqref{eq:PDEhom1a1}, \eqref{eq:PDEhom2a1},
\eqref{eq:PDEmicrohom1a1}, \eqref{eq:PDEmicrohom2a1}, \eqref{eq:Fluxhom1a1}, \eqref{eq:Fluxhom2a1}, \eqref{eq:InitDatahom1a1}, \eqref{eq:InitDatahom2a1} are still in force, with $v,u^B$ having null trace on $\partial\Om\times(0,T)$.
However, as we will see in Remark \ref{r:r8} below, we will find that equation \eqref{eq:PDEhom3a1} becomes
\begin{equation}\label{eq:PDEhom1a1_bis}
\capuno\partial_t[u]+\capdue[u]=0\,,\qquad \hbox{in $\Om_T$,}
\end{equation}
\eqref{eq:PDEmicrohom3a1} and \eqref{eq:Fluxhom3a1} disappear, and
equation \eqref{eq:PDEhom2a1} simplifies to
\begin{multline}\label{eq:PDEhom1a2_bis}
-\Div\left(\int_{\Perout}(\dfint+\dfout)(\nabla u^B\!+\!\nabla_y\hat u^B)\di y\right)
\\
-\Div\left(\int_{\Perout} \dfint(\nabla v^B+\nabla_y \hat v^B)\di y\right)= |\Perout|(f_1-f_2),\qquad \text{in $\Om_T$.}
\end{multline}
Therefore,
the function $u^D$ can be explicitly determined in terms of $u^B$ and $\overline s_1$, i.e.
$$
u^D(x,t)=u^B(x,t)-\left(\frac{1}{|\Permemb|}\int_{\Permemb}\overline s_1(x,y)\di\sigma(y)\right){\rm e}^{-\capdue t/\capuno}\,.
$$
In particular, the damaged zone affects the macroscopic model only through the
physical properties ($\capuno,\capdue$) of the boundary of such a zone, while $\dfdis$ has no influence in the homogenized limit.
\end{remark}

{\it Proof of Theorem \ref{t:hom1a1}.}
In the weak formulation \eqref{eq:weak5}, let us take, as test functions,
\begin{equation*}
\testb=\phi_B(x,t)+\eps\psi_B(x,t,x/\eps)
\qquad\hbox{and}\qquad
\testd=\phi^i_D(x,t)+\eps\psi^i_D(x,t,x/\eps)\,,\ i=1,2\,,
\end{equation*}
where $\phi_B,\phi^1_D,\phi^2_D\in \CC^1(\overline\Om_T)$, with compact support in $\Om$, for every $t\in[0,T]$,
and such that $\phi_B(x,T)=\phi_D(x,T)=0$, for every $x\in\overline\Om$, $\psi_B\in \CC^1(\overline\Om_T;\CC^1_\#(\overline\Perout))$,
with compact support in $\Om$, for every $(t,y)\in[0,T]\times\overline Y$, and such that $\psi_B(x,T,y)=0$, for every $(x,y)\in\overline\Om\times
\overline Y$, $\psi_D=(\psi^1_D,\psi^2_D)\in \CC^1(\Om_T;\X^1_\#(Y))$,
with compact support in $\overline\Om$, for every $(t,y)\in[0,T]\times\overline Y$,
and such that $[\psi_D(x,T,y)]=0$, for every $(x,y)\in\overline\Om\times \overline Y$.
Then, we obtain
\begin{multline*}
-\int_{\Omout_T} \ve(\partial_t\phi_B+\eps\partial_t\psi_B)\di x\di t+\int_{\Omout_T}\dfinte\nabla \ve\cdot(\nabla \phi_B+\eps\nabla_x\psi_B+\nabla_y\psi_B)\di x\di t
\\
+\int_{\Omout_T}\dfinte\nabla \ue\cdot(\nabla \phi_B+\eps\nabla_x\psi_B+\nabla_y\psi_B)\di x\di t
+\int_{\Omout_T} I_{ion}(\ve,\we)(\phi_B+\eps\psi_B)\di x\di t
\\
+\int_{\Omout_T}(\dfinte+\dfoute)\nabla \ue\cdot (\nabla \phi^1_D+\eps\nabla_x\psi_D^1+\nabla_y\psi_D^1)\di x\di t
\\
+\!\!\int_{\Omout_T}\!\dfinte\nabla \ve\cdot(\nabla \phi_D^1+\eps\nabla_x\psi_D^1+\nabla_y\psi_D^1)\di x\di t
+\!\!\int_{\Omint_T}\!\dfdise\nabla  \ue\cdot (\nabla \phi_D^2+\eps\nabla_x\psi_D^2+\nabla_y\psi_D^2)\di x\di t
\end{multline*}
\begin{multline}\label{eq:weak3}
-{\capuno}{\eps}\int_{\Memb_T}[\ue]\partial_t[\phi^D+\eps\psi_D]\di\sigma\di t
+{\capdue}\eps\int_{\Memb_T}[\ue][\phi^D+\eps\psi_D]\di\sigma\di t
\\
=\int_{\Omout_T}f_1(\phi_B+\eps\psi_B)\di x\di t+\int_{\Omout_T}(f_1-f_2)(\phi_D^1+\eps\psi_D^1)\di x\di t
\\
+\int_{\Omout} \overline v_0(\phi_B(0)+\eps\psi_B(0))\di x
+{\capuno\eps}\int_{\Memb}\frac{s_{0\eps}}{\eps}[\psi_D](0)\di\sigma\,,
\end{multline}
where, with a little abuse of notation, we denote by $[\phi^D]=\phi^1_D-\phi^2_D$.
Unfolding and passing to the limit, we arrive at
\begin{multline*}
-|\Perout|\int_{\Om_T}v \partial_t\phi_B\di x\di t+\int_{\Om_T}\int_{\Perout}\dfint(\nabla v+\nabla_y\hat v)\cdot
(\nabla\phi_B+\nabla_y\psi_B)\di y\di x\di t
\\
+\int_{\Om_T}\int_{\Perout}\dfint(\nabla u^B+\nabla_y\hat u^B)\cdot(\nabla\phi_B+\nabla_y\psi_B)\di y \di x\di t
+|\Perout|\int_{\Om_T}I_{ion}(v,w)\phi_B\di x\di t
\\
+\int_{\Om_T}\int_{\Perout}(\dfint+\dfout)(\nabla u^B+\nabla_y\hat u^B)\cdot(\nabla\phi_D^1+\nabla_y\psi_D^1)\di y\di x\di t
\\
+\int_{\Om_T}\int_{\Perout}\dfint(\nabla v+\nabla_y\hat v)\cdot(\nabla\phi_D^1+\nabla_y\psi_D^1)\di y\di x\di t
\\
+\int_{\Om_T}\int_{\Perint}\dfdis(\nabla u^D+\nabla_y\hat u^D)\cdot(\nabla \phi_D^2+\nabla_y\psi_D^2)\di y\di x \di t
\end{multline*}
\begin{multline}\label{eq:a9a1}
-\capuno\int_{\Om_T}\int_{\Permemb}[u]\partial_t[\phi_D]\di\sigma\di x \di t
+\capdue\int_{\Om_T}\int_{\Permemb}[u][\phi_D]\di\sigma\di x\di t
\\
= |\Perout|\int_{\Om_T}f_1\phi_B\di x\di t+|\Perout|\int_{\Om_T}(f_1-f_2)\phi^1_D\di x\di t
\\
+|\Perout|\int_{\Om}\overline v_0\phi_B(0)\di x+\capuno\int_\Om\int_\Permemb \overline s_1[\phi_D](0)\di\sigma\di x\,,
\end{multline}
where we have used Lemmas \ref{l:conv}, \ref{l:conv1} and \ref{l:conv3}.
In order to get the strong formulation \eqref{eq:PDEhom1a1}--\eqref{eq:BoundDatahoma1}, we localize \eqref{eq:a9a1}, taking first
$\psi_B=\phi_D^1=\phi_D^2=\psi_D^1=\psi_D^2=0$, then $\phi_B=\phi^2_D=\psi_B=\psi_D^1=\psi_D^2=0$ and finally
$\phi_B=\phi^1_D=\psi_B=\psi_D^1=\psi_D^2=0$, so that we arrive at \eqref{eq:PDEhom1a1}--\eqref{eq:PDEhom3a1}
and \eqref{eq:InitDatahom1a1}, \eqref{eq:InitDatahom2a1}. Moreover, we take $\phi_B=\phi_D^1=\phi_D^2=\psi_D^1=\psi_D^2=0$, which gives
\eqref{eq:PDEmicrohom1a1} and \eqref{eq:Fluxhom1a1}.
In the next step, we take first $\phi_B=\phi^1_D=\phi^2_D=\psi_B=\psi_D^2=0$ and
then $\phi_B=\phi^1_D=\phi^2_D=\psi_B=\psi_D^1=0$, in order to obtain
\eqref{eq:PDEmicrohom2a1}, \eqref{eq:PDEmicrohom3a1} and \eqref{eq:Fluxhom2a1}, \eqref{eq:Fluxhom3a1}.
The boundary condition \eqref{eq:BoundDatahoma1} is a direct consequence of the fact that $v, u^1_D, u^2_D\in L^2(0,T;H^1_0(\Om))$.

Finally, the limit gating problem \eqref{eq:gating3}--\eqref{eq:gating4} and the uniqueness for the two-scale homogenized system
are obtained as in the proof of Theorem \ref{t:hom1}.
\hfill$\Box$

\begin{thm}\label{t:t1a1}
Assume to be in the connected/connected geometry. Then, the two-scale system \eqref{eq:PDEhom1a1}--\eqref{eq:BoundDatahoma1}
can be rewritten as the single-scale degenerate parabolic system given by
\begin{equation}\label{eq:a15a1}
\begin{aligned}
\partial_t v-\Div\left(\Ahomuno\nabla (v+u^B)\right)+I_{ion}(v,w)=f_1\,,\qquad & \hbox{in $\Om_T$;}
\\
-\Div\left(\Ahomuno\nabla (u^B+v)\right)-\Div\left(\Ahomdue_B\nabla u^B+\Ahomdue_D\nabla u^D\right)
=f_1-f_2\,,\qquad & \hbox{in $\Om_T$;}
\\
\frac{\capuno|\Permemb|}{|\Perout|}\partial_t[u]+\frac{\capdue|\Permemb|}{|\Perout|}[u]+\Div\left(\Ahomdue_D\nabla u^D\right)
=0\,,\qquad & \hbox{in $\Om_T$,}
\end{aligned}
\end{equation}
complemented with the initial and the boundary conditions \eqref{eq:InitDatahom1a1}--\eqref{eq:BoundDatahoma1} and the gating problem
\eqref{eq:gating3}--\eqref{eq:gating4}, where the matrix $\Ahomuno$ is defined in \eqref{eq:matrix2} and $\Ahomdue_B,\Ahomdue_D$ are defined in \eqref{eq:matrix2_a0} and \eqref{eq:matrix3_a0}, respectively.
\end{thm}

\begin{proof}
As in Subsection \ref{ss:scaling1}, thanks to \eqref{eq:PDEmicrohom1a1} and \eqref{eq:Fluxhom1a1}, we can factorize
\begin{equation}\label{eq:factor1a1}
(\hat v+\hat u^B)(x,y,t)=-\zeta(y)\cdot\nabla(v+u^B)(x,t)\,,
\end{equation}
where the cell functions $\zeta=(\zeta^1,\dots,\zeta^N)$, with $\zeta^j\in H^1_{\#}(\Perout)$ and $\mathcal{M}_{\Perout}(\zeta^j)=0$,
are the solutions of the cell problem \eqref{eq:cell1}.
Moreover, taking into account \eqref{eq:PDEmicrohom2a1}, \eqref{eq:Fluxhom2a1} and \eqref{eq:PDEmicrohom3a1}, \eqref{eq:Fluxhom3a1}, we can factorize
\begin{equation}\label{eq:factor2a1}
\hat u^B(x,t,y)=-\widehat\chi_0^B(y)\cdot \nabla u^B(x,t),
\qquad \hat u^D(x,t,y)=-\widehat\chi_0^D(y)\cdot \nabla u^D(x,t),
\end{equation}
where, for $j=1,\dots,N$, $\widehat\chi_0^{B,j}\in H^1_\#(\Perout)$, $\widehat\chi_0^{D,j}\in H^1_\#(\Perint)$, $\mathcal{M}_{\Perout}(\widehat\chi_0^{B,j})=0
=\mathcal{M}_{\Perint}(\widehat\chi_0^{D,j})$ are the solutions of \eqref{eq:cell2_a0}.
Inserting \eqref{eq:factor1a1} and \eqref{eq:factor2a1} in \eqref{eq:PDEhom1a1}--\eqref{eq:PDEhom3a1}, we get the single-scale
homogenized degenerate parabolic system \eqref{eq:a15a1},
where the matrices $\Ahomuno,\,\Ahomdue_B,\,\Ahomdue_D$ are defined in \eqref{eq:matrix2}, \eqref{eq:matrix2_a0} and \eqref{eq:matrix3_a0},
respectively.
\end{proof}

\begin{remark}\label{r:r8}
In the connected/disconnected geometry, the system \eqref{eq:PDEhom1a1}--\eqref{eq:InitDatahom2a1}, with $u^D$ (which
now is only an $L^2(\Om_T)$-function) replaced by $u^B$
in \eqref{eq:PDEhom3a1}, \eqref{eq:PDEmicrohom3a1} and \eqref{eq:Fluxhom3a1}, and $v,u^B$ having null trace on the boundary
$\partial\Om\times(0,T)$, is still in force.
Then, in \eqref{eq:factor2a1}, the factorization of $\hat u^D$ is replaced by $\hat u^D=-\widehat\chi^{D}_0\cdot\nabla u^B$,
with $\widehat\chi^{D}_0$ as above.
However, as in Subsection \ref{ss:scaling2}, we obtain that $\widehat\chi_0^{D,j}(y)=y^j$, up to an additive constant.
This implies $\nabla u^B+\nabla_y\hat u^D=0$ and, hence, \eqref{eq:PDEhom3a1} is replaced by \eqref{eq:PDEhom1a1_bis}
and equations \eqref{eq:PDEmicrohom3a1} and \eqref{eq:Fluxhom3a1}, actually, disappear. Moreover, equation \eqref{eq:PDEhom2a1} simplifies
in equation \eqref{eq:PDEhom1a2_bis} and, finally, the matrix $\Ahomdue_D=0$.
Therefore, the single-scale degenerate parabolic system \eqref{eq:a15a1} becomes
\begin{equation}\label{eq:a15bis}
\begin{aligned}
\partial_t v-\Div\left(\Ahomuno\nabla (v+u^B)\right)+I_{ion}(v,w)=f_1\,,\qquad & \hbox{in $\Om_T$;}
\\
-\Div\left(\Ahomuno\nabla (u^B+v)\right)-\Div\left(\Ahomdue_B\nabla u^B\right)
=f_1-f_2\,,\qquad & \hbox{in $\Om_T$;}
\\
\capuno\partial_t[u]+\capdue[u]=0\,,\qquad & \hbox{in $\Om_T$.}
\end{aligned}
\end{equation}
\end{remark}

\begin{remark}\label{r:r11}
Notice that the limit problem \eqref{eq:a15a1}, in the connected/connected case, and the limit problem \eqref{eq:a15bis},
in the connected/disconnected case, both lead to a kind of tridomain model. Indeed,
similarly as in Remark \ref{r:r9}, let us denote by $u^B_1,u^B_2$ and $u^D$, respectively,
the limits of the functions $\ubuno,\ubdue$ and $\ud$,
appearing in the system \eqref{eq:PDEin}--\eqref{eq:InitData3}. Recalling that $\ve=\ubuno-\ubdue$,
$\ue\mid_{\Omout_{T}}=\ubdue$ and $\ue\mid_{\Omint_{T}}=\ud$, we can replace $v=u^B_1-u^B_2$ and
$u^B=u^B_2$ in \eqref{eq:a15a1}, thus obtaining
\begin{equation*}
\begin{aligned}
\partial_t (u^B_1-u^B_2)-\Div\left(\Ahomuno\nabla (u^B_1)\right)+I_{ion}(u^B_1-u^B_2,w)=f_1\,,\qquad & \hbox{in $\Om_T$;}
\\
-\Div\left(\Ahomuno\nabla u^B_1\right)-\Div\left(\Ahomdue_B\nabla u^B_2+\Ahomdue_D\nabla u^D\right)
=f_1-f_2\,,\qquad & \hbox{in $\Om_T$;}
\\
\frac{\capuno|\Permemb|}{|\Perout|}\partial_t(u^B_2-u^D)+\frac{\capdue|\Permemb|}{|\Perout|}(u^B_2-u^D)+\Div\left(\Ahomdue_D\nabla u^D\right)
=0\,,\qquad & \hbox{in $\Om_T$.}
\end{aligned}
\end{equation*}
Analogously, \eqref{eq:a15bis} becomes
\begin{equation*}
\begin{aligned}
\partial_t (u^B_1-u^B_2)-\Div\left(\Ahomuno\nabla (u^B_1)\right)+I_{ion}(u^B_1-u^B_2,w)=f_1\,,\qquad & \hbox{in $\Om_T$;}
\\
-\Div\left(\Ahomuno\nabla u^B_1\right)-\Div\left(\Ahomdue_B\nabla u^B_2\right)
=f_1-f_2\,,\qquad & \hbox{in $\Om_T$;}
\\
\frac{\capuno|\Permemb|}{|\Perout|}\partial_t(u^B_2-u^D)+\frac{\capdue|\Permemb|}{|\Perout|}(u^B_2-u^D)
=0\,,\qquad & \hbox{in $\Om_T$.}
\end{aligned}
\end{equation*}
\end{remark}


\end{document}